\documentclass[hidelinks,onefignum,onetabnum]{siamart251216}

\usepackage{amsmath,amsfonts,amssymb,bm}
\usepackage{graphicx}
\usepackage{booktabs}
\usepackage{cleveref}
\usepackage{todonotes,mathtools} %, relsize,geometry}
\usepackage{algorithm}
\usepackage{algpseudocode}
\usepackage{subcaption}
\usepackage{graphicx}
\usepackage{multirow}
\usepackage{multirow}

\newcommand{\KMN}{{\mbox{\sc kme}}}
\newcommand{\KML}{{\mbox{\sc kml}}}
\newtheorem{example}[theorem]{Example}
\newtheorem{remark}[theorem]{Remark}

\def\im{i_m}
\def\imm{i_{m+1}}

\def\il{i_\ell}

\def\R{\mathbb{R}}
\def\N{\mathbb{N}}
\DeclareMathOperator{\K}{K}
\newcommand{\rev}[1]{\textcolor{black}{#1}}
\newcommand{\vs}[1]{\textcolor{black}{#1}}
\newsavebox{\fitbox}
\newcommand{\fittable}[1]{%
  \sbox{\fitbox}{#1}%
  \ifdim\wd\fitbox>\linewidth
    \resizebox{\linewidth}{!}{\usebox{\fitbox}}%
  \else
    \usebox{\fitbox}%
  \fi
}

\title{Edge-based Katz centralities for spatio-temporal multiplex networks}
\author{Kai Bergermann\thanks{Centro di Ricerca Matematica Ennio De Giorgi, Scuola Normale Superiore, Pisa, Italy, {\tt kai.bergermann@sns.it}} \and Francesco Gravili\thanks{Work done while at Dipartimento di Matematica,
Alma Mater Studiorum Universit\`a di Bologna,
Piazza di Porta San Donato  5, I-40127 Bologna, Italy} \and Valeria Simoncini\thanks{Dipartimento di Matematica and (AM)$^2$,
Alma Mater Studiorum Universit\`a di Bologna,
Piazza di Porta San Donato  5, I-40127 Bologna, Italy, and IMATI-CNR, Pavia, Italy,  {\tt valeria.simoncini@unibo.it}} \and Martin Stoll\thanks{Chair of Scientific Computing, Technische Universit\"at Chemnitz, Department of Mathematics, 09107 Chemnitz, Germany, {\tt martin.stoll@math.tu-chemnitz.de}}}

\begin{document}
\maketitle

\begin{abstract}
Katz centrality is a well-established measure to identify and rank the most important nodes in complex networks by means of a linear system solve.
Recent works have developed notions of Katz centrality for temporal, i.e., time-evolving networks.
Their drawback is that small changes in the network structure may drastically change centralities across large parts of the network.
Moreover, an unproportional effort in the network science community has been devoted to the study of node-based quantities while edge-based measures are far less explored.

In this manuscript, we introduce a novel failure-aware temporal multiplex network model for computing edge-based Katz centralities for spatial networks.
This class of networks admits the use of the line graph as network representation.
We use a block-triangular and -banded supra-adjacency matrix representation, modeling inter-layer connections as low-rank matrices, which assigns a spatially constrained emphasis on up- or downdated edges.
These could represent, e.g., \vs{breakages of water / gas pipes} or \vs{obstructions in street networks 
%in water or street networks 
in} real-time infrastructure monitoring applications.
We analyze the structure of the %inverse of the 
block matrix \vs{inverse} with respect to entry decay in the blocks above the block-diagonal.
This gives rise to two truncation approaches that allow drastic computational runtime reductions at the cost of introducing a controlled  truncation error.
Numerical experiments on a range of real-world spatio-temporal networks of size up to $2\cdot 10^8$ illustrate accuracy and efficiency with runtime gains of up to a factor of $679$.
\end{abstract}

%%%%%%%%%%%%%%%%%%%%%%%%%

\section{Introduction}

Complex networks have become indispensable for the study of complex systems from \vs{social,
%over
biological} to infrastructural applications \cite{newman2003structure,boccaletti2006complex}.
Temporal networks describe the evolution of such network structures over time.
A popular approach to capture time-varying networks is via temporal multiplex networks \cite{holme2012temporal,kivela2014multilayer,boccaletti2014structure,masuda2016guide}.

Centrality measures rank among the many well-studied structural properties of complex networks \cite{freeman1977set,freeman1978centrality,bonacich1987power,brin1998anatomy,page1999pagerank}.
A recent survey lists over $400$ different ways of measuring the importance of nodes in a complex network \cite{shvydun2025zoo}.
Matrix function-based centrality measures form a family of such measures based on counting walks on networks that allow an elegant and numerically attractive formulation in terms of matrix functions \cite{katz1953new,estrada2005subgraph,estrada2008communicability,estrada2010network,benzi2013total,benzi2020matrix}.

A successful model of time varying networks relies on a multilayer graph, giving rise to a multiplex matrix that encompasses varying adjacency matrices \cite{kivela2014multilayer,boccaletti2014structure}.
The analysis of network importance has progressed significantly from static node centrality to dynamic measures capable of capturing structural evolution \cite{grindrod2011communicability,grindrod2013matrix,grindrod2014dynamical,fenu2017block,arrigo2017sparse,al2021block,arrigo2022dynamic}. Block matrix formulations provide a crucial framework for modeling (time) evolving networks~\cite{fenu2017block}, which has been extended to layer-coupled multiplex structures to handle inter-layer dependencies~\cite{arrigo2016edge, el2023perron, noschese2024enhancing}. While traditional measures focus on nodes, edge centrality has emerged as a critical tool for assessing robustness, often utilizing line graphs~\cite{de2020edge, de2021centrality} or sensitivity analysis via Fr\'echet derivatives to quantify the impact of edge modifications on total communicability~\cite{de2022communication, schweitzer2023sensitivity, noschese2025edge,bergermann2022fast}. Recent work has formalized edge importance in complex networks through Perron sensitivity and hub/authority communicabilities~\cite{arrigo2016edge, noschese2024enhancing}. To mitigate the computational cost of recomputing centralities following network changes, updating strategies leveraging walk counting have been proposed~\cite{arrigo2025updating}, alongside optimization approaches for maximizing network robustness~\cite{massei2024optimizing}. Furthermore, temporal extensions of Katz centrality have been explored~\cite{taylor2021tunable}, yet a unified approach combining edge-based metrics with efficient updating for failure-aware temporal multiplex networks remains largely unexplored.

% \begin{itemize}
% \item Edge centralities via line graphs \cite{de2020edge,de2021centrality}
% \item Frechet derivatives for measuring the sensitivity of edge addition/removal on the total communicability of networks \cite{de2022communication,schweitzer2023sensitivity,noschese2025edge}
% \item Perron sensitivity \cite{de2022communication,noschese2025edge}
% \item generalizations to multiplex/multilayer networks \cite{el2023perron,noschese2024enhancing}
% \item Definition of total hub/authority communicabilities $\bm{1}^\top f(\bm{AA}^\top)\bm{1}$ and $\bm{1}^\top f(\bm{A}^\top\bm{A})\bm{1}$ for digraphs.
% In-/decrease of communicabilities via the up-/downdate of central edges.
% The proposed edge centralities (various measures) are defined as the product of node centralities of the adjacent nodes. \cite{arrigo2016edge}
% \item Investigation of the change of Katz centralities under node/edge removal by counting walks -- analysis heavily relies on series representation \cite{arrigo2025updating} 
% \item Maximizing in-/decrease of network robustness by several optimization approaches \cite{massei2024optimizing}
% \item several works on temporal Katz centralities... % \cite{Grindrod2011}
% \end{itemize}

\subsection{Contributions}
In this work, we propose a failure-aware temporal multiplex network model for computing edge-based temporal Katz centralities for spatial networks.
Our contributions are summarized as follows.
\begin{itemize}
  \item We introduce a novel edge-based block-matrix representation \eqref{eq:block_matrix} of temporal multiplex networks in which network up- or downdates are incorporated via low-rank inter-layer coupling matrices in the supra-adjacency matrix.
  This puts a localized emphasis on up- or downdated edges and their direct neighborhoods, 
  %while
  \vs{whereas}
  centralities in the standard temporal multiplex model \eqref{eq:block_matrix_standard} may change substantially in large parts of the network.
  \item We analyze entry decay in the upper-triangular blocks of the failure-aware supra-adjacency matrix when Katz centralities are computed as row sums of the resolvent matrix function.
  In \Cref{sec:structural_analysis}, we show that entry decay depends exponentially on the shortest-path distance between up- or downdated edges in consecutive time layers.
  Entry decay is expected to be rapid since spatial networks with low average degree and large diameter are characterized by large average distances.
  \item Based on our analysis, we propose a computationally efficient truncation strategy to approximate failure-aware edge-based temporal Katz centralities.
  Reducing computations to the evaluation of the diagonal blocks allows for a parallel approach that substantially outperforms the solution of the full block system or an inherently sequential back-substitution strategy.
  \item We study dynamically updated networks that appear, e.g., in real-time monitoring of infrastructure networks.
  Adding a new time layer generally affects centralities across all previous time steps and calls for the re-computation of the full block system.
  Again based on our analysis, \Cref{sec:truncation_strategy} introduces a second truncation strategy that reduces the computational effort to the solution of a small number of small linear systems.
\end{itemize}

%%%%%%%%%%%%%%%%%%%%%%%%%
\subsection{Notation}
In the following, bold Roman letters denote matrices associated with the graph dimensions, while curly bold letters refer to multiplex matrices. The vector ${\bm e}_j$ denotes the $j$-th vector of the canonical basis, whose dimension is clear from the context. The vector of all ones is denoted with ${\bm 1}$, while the zero matrix is denoted with $\bm{0}\in\mathbb{R}^{N\times M}$; in both cases the dimensions should be clear from the context, or explicitly given as a subscript, e.g., ${\bm 1}_L\in{\mathbb R}^L$. With some abuse of notation, subscripts and superscripts are also used to identify quantities at the indicated step. For instance, ${\bm A}_L$ refers to a matrix
at time step $L$, while ${\bm c}_\ell^{(L)}$ identifies the
$\ell$-th vector in a sequence at time step $L$.

%%%%%%%%%%%%%%%%%%%%%%%%%
\section{Models for temporal multiplex networks}\label{sec:multiplex_models}

We start by introducing the algebraic representation of edge-based networks.
After briefly describing the current standard block matrix model of temporal multiplex networks, we introduce a novel failure-aware temporal multiplex model, which puts a more localized emphasis on the time-varying parts of the network.

\subsection{Edge-based networks}\label{sec:edge_networks}

A graph or network $\mathcal{G}=(\mathcal{V},\mathcal{E})$ consists of a set $\mathcal{V}$ of $n\in\N$ nodes and a set $\mathcal{E}\subset\mathcal{V}\times\mathcal{V}$ of $N\in\N$ edges.
The majority of network analysis techniques such as centrality analysis or community detection is posed on the node-level \cite{newman2003structure,boccaletti2006complex,kivela2014multilayer,boccaletti2014structure}.
However, in spatial networks such as road, water, power, or gas networks it can be argued that the objects of interest are \emph{edges} along which passengers or utilities are distributed.

The standard algebraic representation of a network $\mathcal{G}$ on the node-level is its adjacency matrix $\bm{A}_{\mathrm{node}}\in\R^{n\times n}$, where $[\bm{A}_{\mathrm{node}}]_{ij}>0$ indicates the presence of an edge connecting nodes $i$ and $j$ while $[\bm{A}_{\mathrm{node}}]_{ij}=0$ indicates that no such edge is present.
For simplicity, we restrict ourselves to undirected and unweighted networks with node adjacency matrices $\bm{A}_{\mathrm{node}}^\top=\bm{A}_{\mathrm{node}}\in\{0,1\}^{n\times n}$.
However, all results and techniques should extend to weighted undirected networks in a straightforward manner.

As network representation on the edge-level, we choose the line graph corresponding to $\bm{A}_{\mathrm{node}}$.
This choice may be inappropriate for certain types of networks such as small-world networks that typically exhibit $N\gg n$ and contain a significantly larger number of non-zero entries in their line graphs.
In our setting of spatial networks $\mathcal{G}$ with an underlying two-dimensional structure, however, we encounter line graphs of similar size and sparsity compared to the original node-network.
The edge adjacency matrix is then defined as $\widetilde{\bm{A}} = \bm{B}^\top\bm{B} - 2\bm{I}\in\R^{N\times N}$, where $\bm{B}\in\{0,1\}^{n\times N}$ denotes the unweighted incidence matrix corresponding to $\bm{A}_{\mathrm{node}}$.

\subsection{The standard temporal multiplex network model}\label{sec:standard_model}

Multiplex networks consist of several network layers that encompass different relationships among the same set of nodes.
Their algebraic representation typically relies on tensors or block matrices \cite{de2013mathematical,kivela2014multilayer,boccaletti2014structure}.
In this work, we adopt the notion of supra-adjacency matrices \cite{kivela2014multilayer}.

Temporal networks may be modeled via multiplex networks by recording the network state at each point in time as one layer \cite{taylor2017eigenvector,taylor2019supracentrality,taylor2021tunable,bergermann2022fast}.
Denoting the number of layers or time points by $L\in\N$, each layer is represented by an edge adjacency matrix $\widetilde{\bm{A}}_\ell\in\R^{N\times N}$, where $\ell=1,\dots,L$, cf.~\Cref{sec:edge_networks}.
For ease of notation, we define the damping parameter
\begin{equation}\label{eq:def_alpha}
\alpha = \frac{c_\alpha}{\max_{\ell=1,\dots,L} \rho(\widetilde{\bm{A}}_\ell)}
\end{equation}
where $c_\alpha\in(0,1)$ and $\rho(\cdot)$ denotes the spectral radius, and absorb it into the definition of the scaled edge adjacency matrices $\bm{A}_\ell = \alpha \widetilde{\bm{A}}_\ell$  for $\ell=1,\dots,L$. 

The standard temporal supra-adjacency matrix is then defined as
\begin{equation}
\label{eq:block_matrix_standard}
\widetilde{\bm{\mathcal{M}}}_L=
\begin{bmatrix}
\bm{A}_1&\bm{I}&\bm{0}&\dots&\bm{0}\\
\bm{0}&\bm{A}_2&\bm{I}&\dots&\bm{0} \\
\bm{0}&\bm{0}&\bm {A}_3 & \ddots&\vdots \\
\vdots & \vdots&\dots & \ddots&\bm{I} \\
\bm{0} & \bm{0} &\dots&\dots& \bm{ A}_L
\end{bmatrix}
\in \mathbb{R}^{NL\times NL},
\end{equation}
with the layer adjacency matrices on the block-diagonal and identity matrices on the super-diagonal.
It has been shown in \cite{fenu2017block} that \eqref{eq:block_matrix_standard} is the block matrix formulation corresponding to the widely studied dynamic communicabilities, cf., e.g., \cite{grindrod2011communicability,fenu2017block,arrigo2017sparse,al2021block,arrigo2022dynamic}.

%\begin{figure}[t]
%\includegraphics[width=.99\textwidth]{figs/Marina.png}
%\caption{Edge-based marginal node Katz centralities of the failure-aware temporal multiplex network model in comparison to the standard temporal multiplex network model.
%In the cycling network in the harbour region of Marina di Pisa, Italy, the most central edge is downdated from time step 1 to time step 2 and restored from time step 2 to time step 3.
%The failure-aware model puts a local emphasis on the up- and downdates.
%In comparison to the standard model, changes in edge centralities are spatially constrained.
%Moreover, the restoration of original edge centrality scores takes relatively many time layers.\todo[inline]{KB to prettify plots with basemap background}}\label{fig:illustration_marina}
%\end{figure}

\begin{figure}[htbp]
	\centering
	\begin{tabular}{ccc}
		\subfloat[Initial Katz centralities]{\includegraphics[width=0.3\linewidth,height=0.20\linewidth]{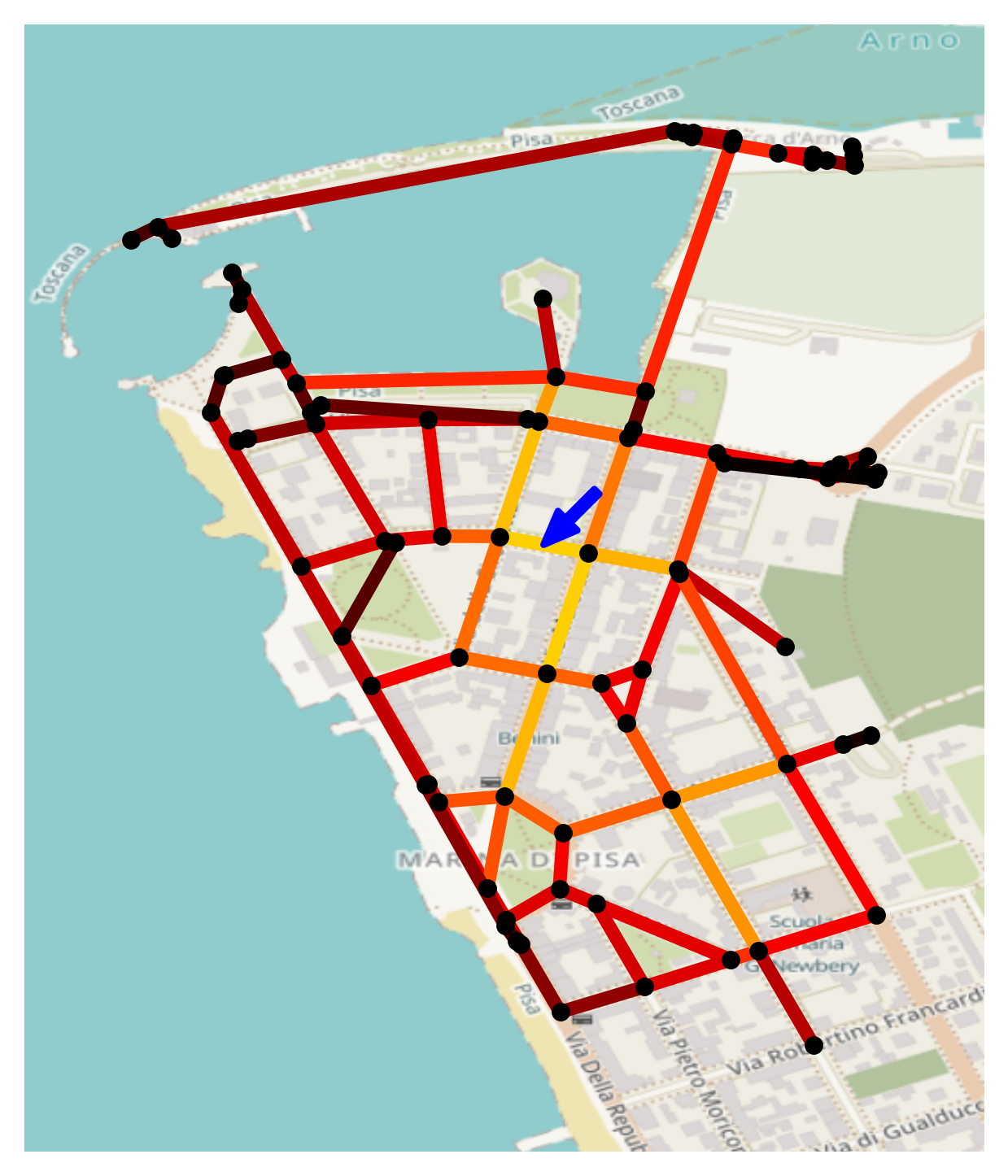}} &
		\subfloat[Standard temporal Katz, $\ell=2$]{\includegraphics[width=0.3\linewidth,height=0.20\linewidth]{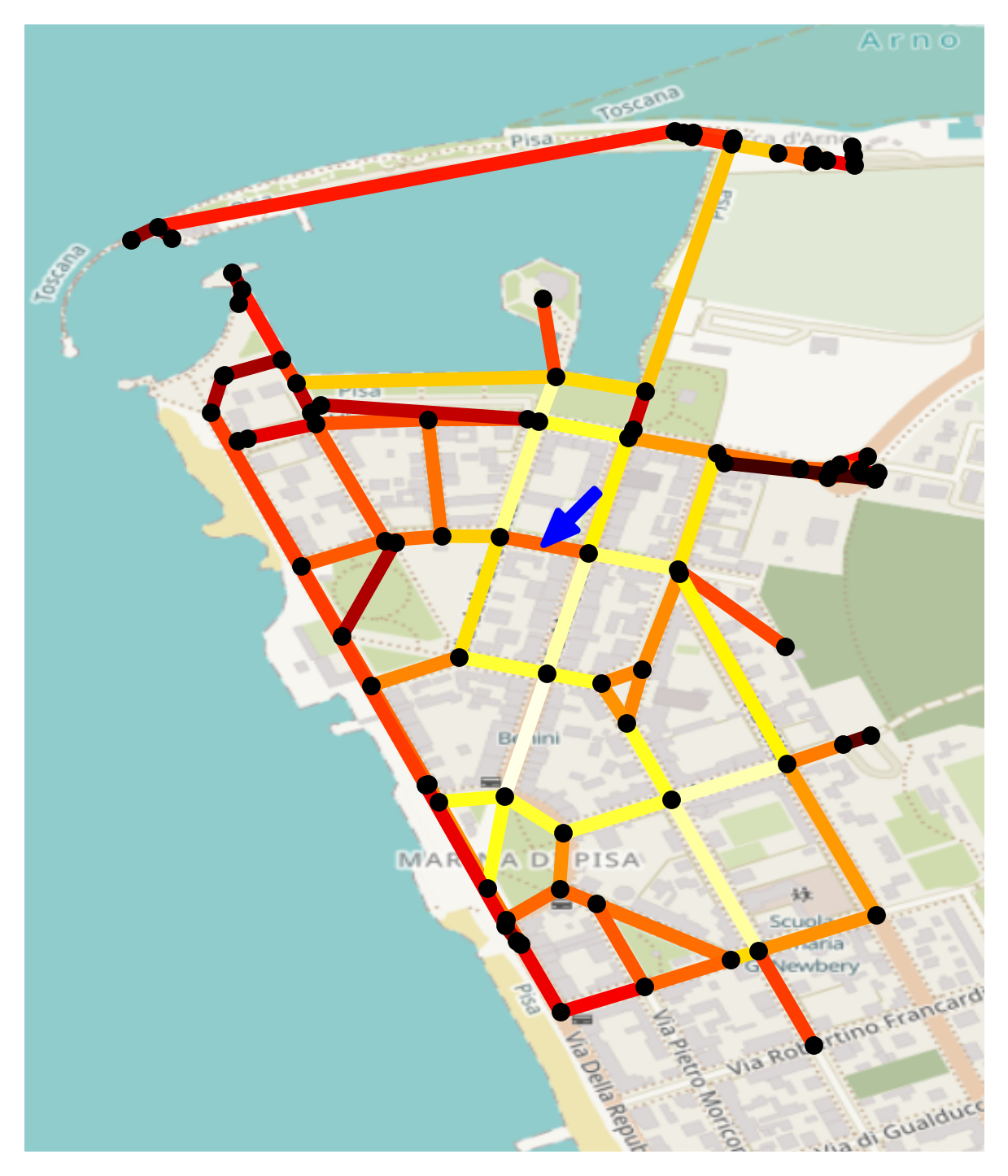}} &
		\subfloat[Failure-aware temporal Katz, $\ell=2$]{\includegraphics[width=0.3\linewidth,height=0.20\linewidth]{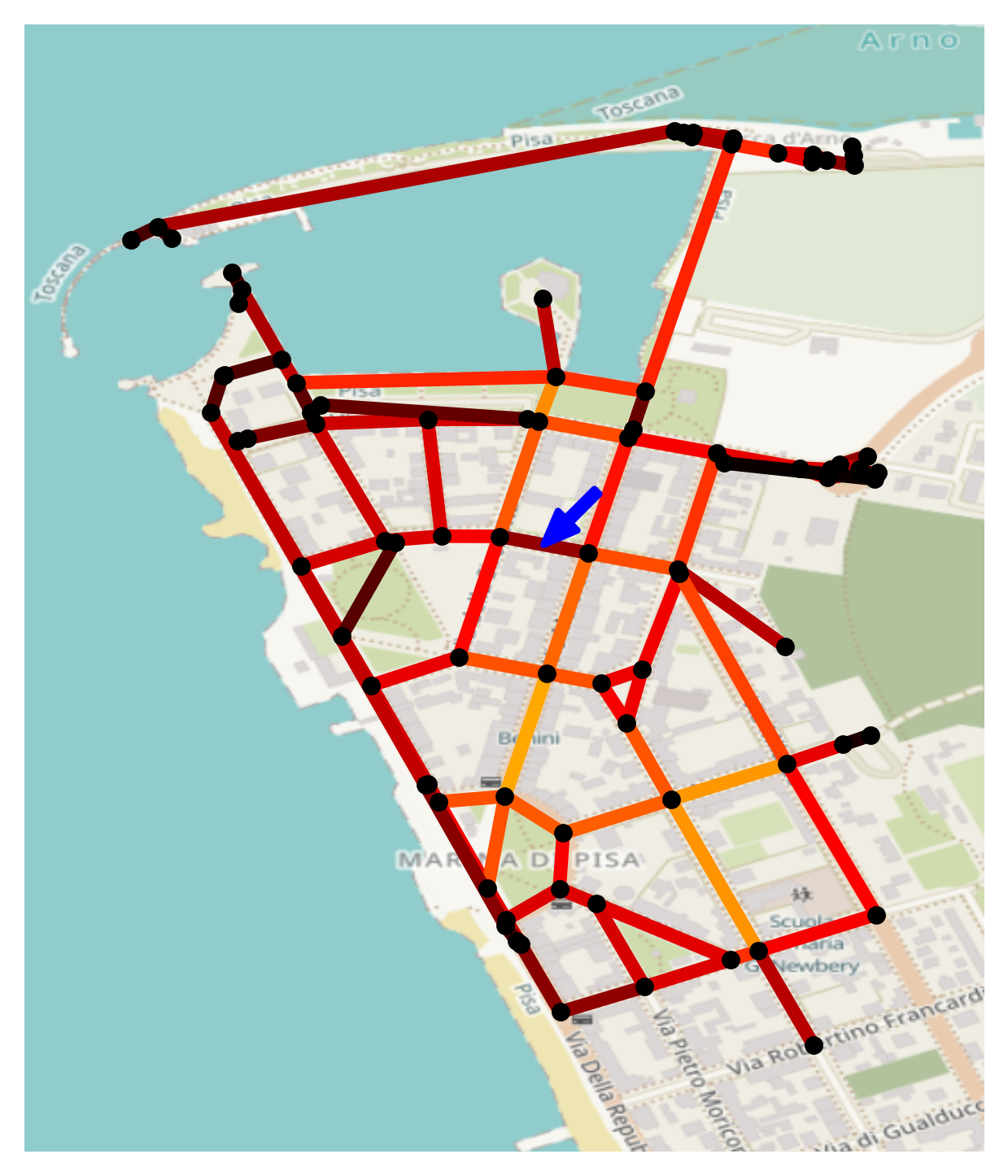}} \\
		%\multirow{2}{*}{\raisebox{-100pt}[0pt][0pt]{
        \includegraphics[height=0.25\linewidth,keepaspectratio]{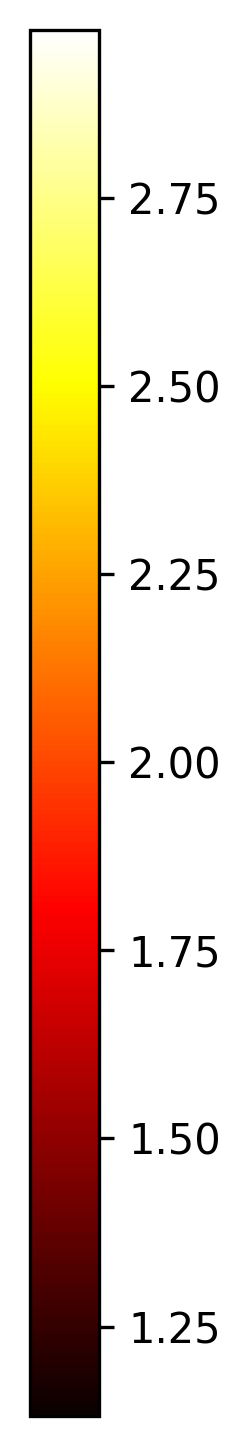}%}}
        &
		\subfloat[Standard temporal Katz, $\ell=5$]{\includegraphics[width=0.3\linewidth,height=0.20\linewidth]{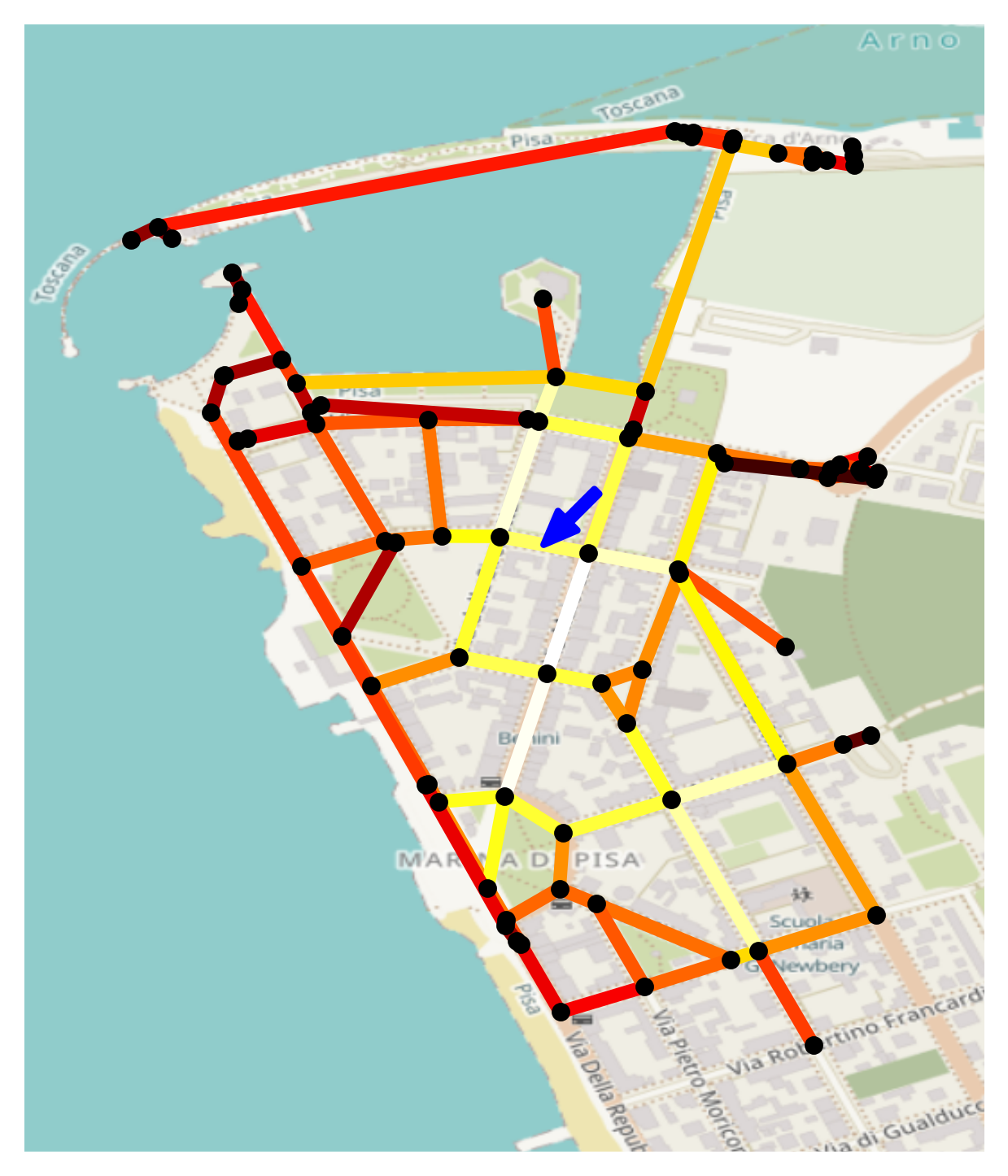}} &
		\subfloat[Failure-aware temporal Katz, $\ell=5$]{\includegraphics[width=0.3\linewidth,height=0.20\linewidth]{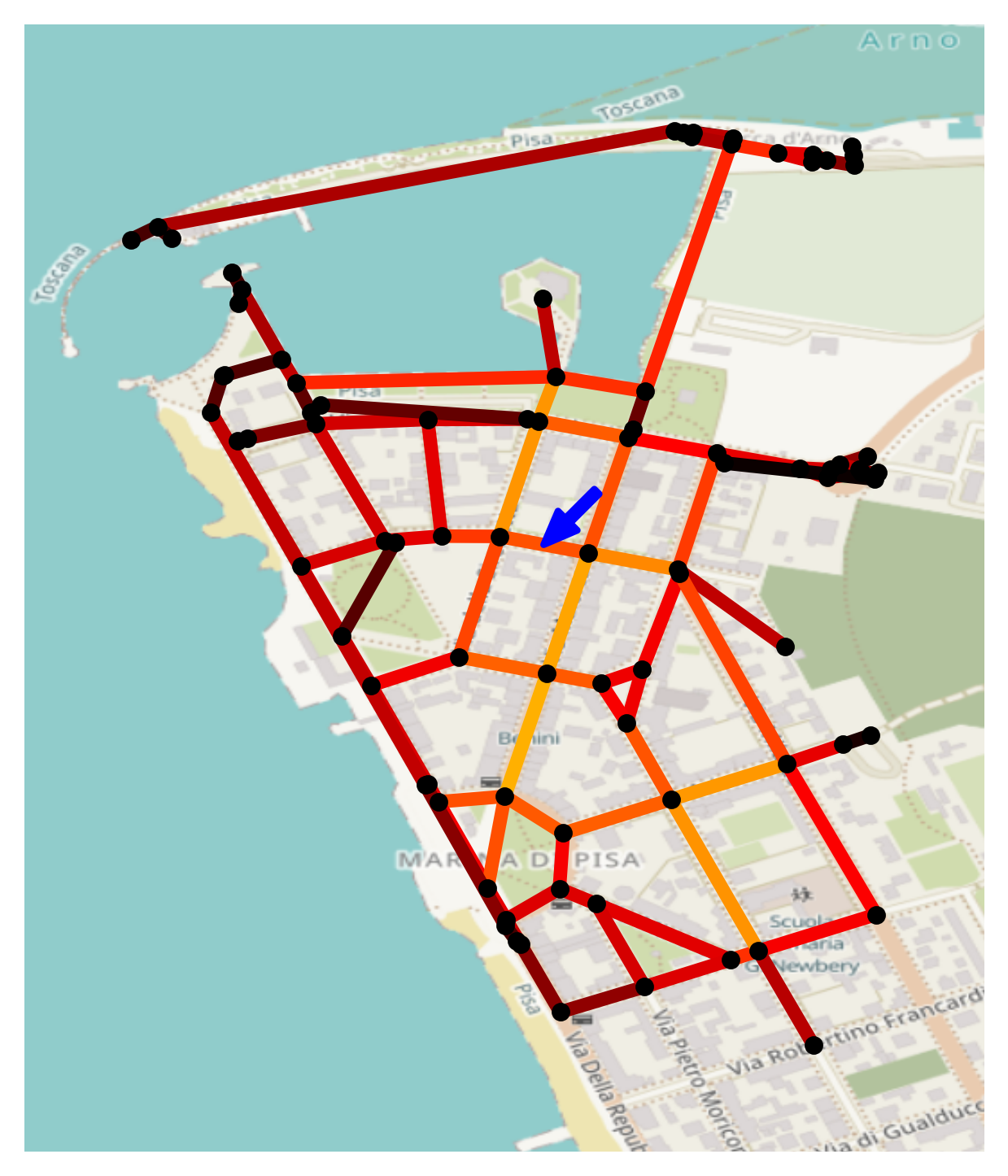}} \\
		% second row of multirow left cell left empty (no caption)
        \raisebox{-20pt}[0pt][0pt]{\includegraphics[height=0.25\linewidth,keepaspectratio]{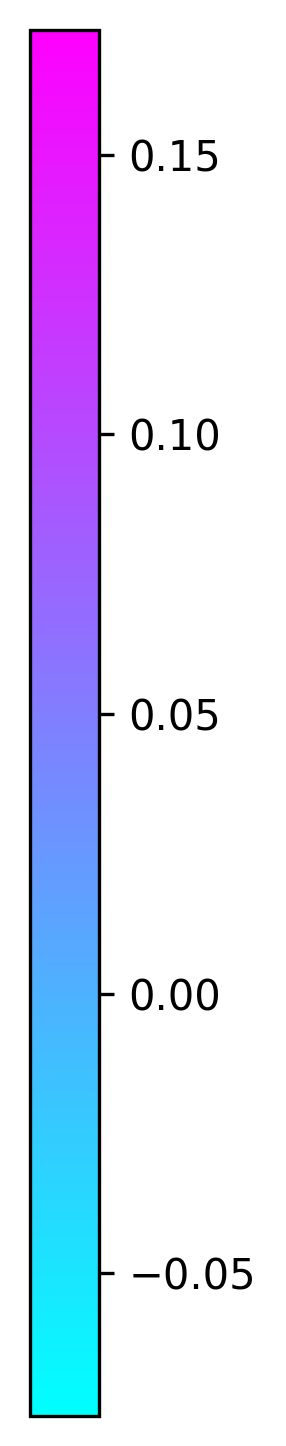}}
		& \subfloat[Difference of standard temporal Katz between $\ell=5$ and $\ell=10$]{\includegraphics[width=0.3\linewidth,height=0.20\linewidth]{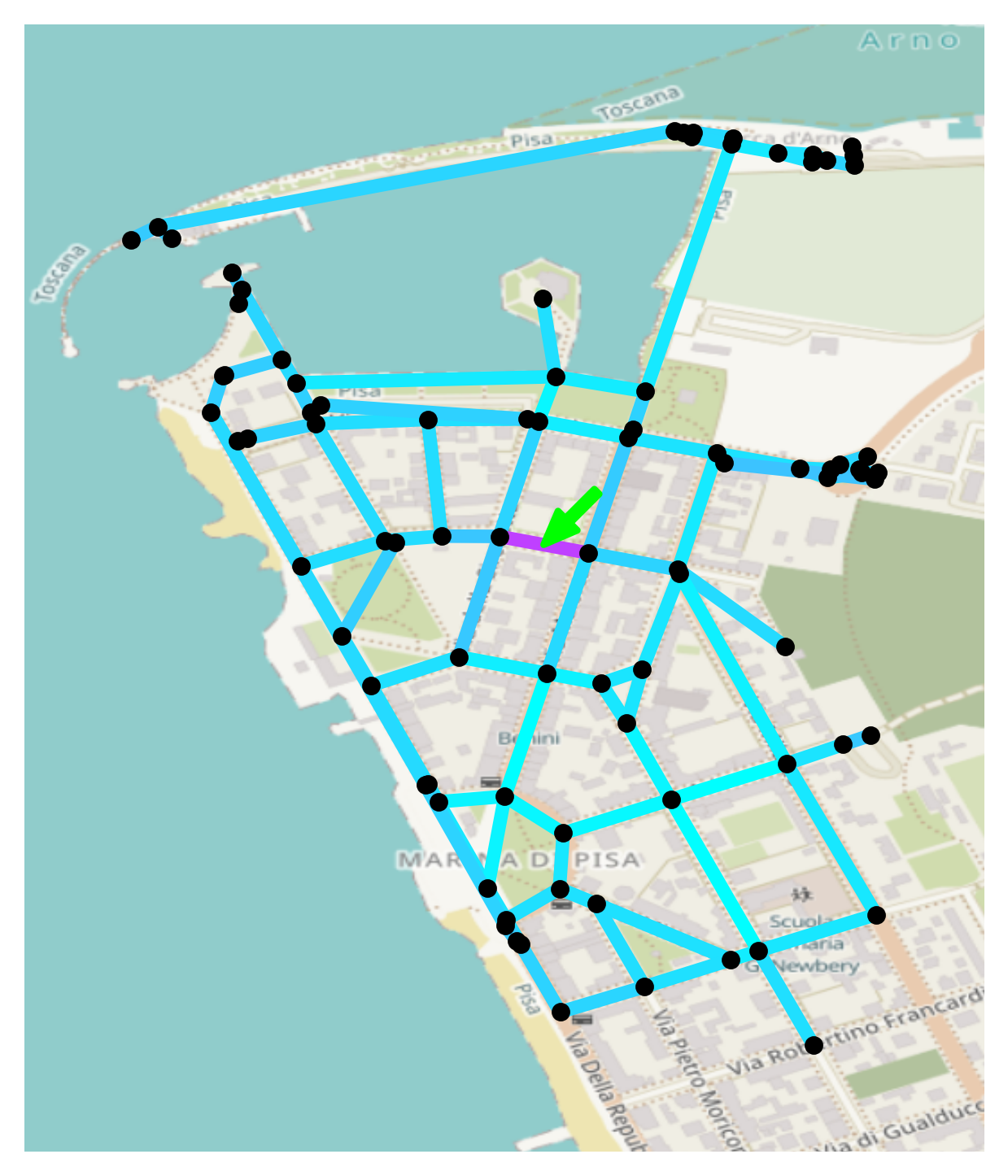}\label{fig:illustration_marina_diff_std}} &
		\subfloat[Difference of failure-aware temporal Katz between $\ell=5$ and $\ell=10$]{\includegraphics[width=0.3\linewidth,height=0.20\linewidth]{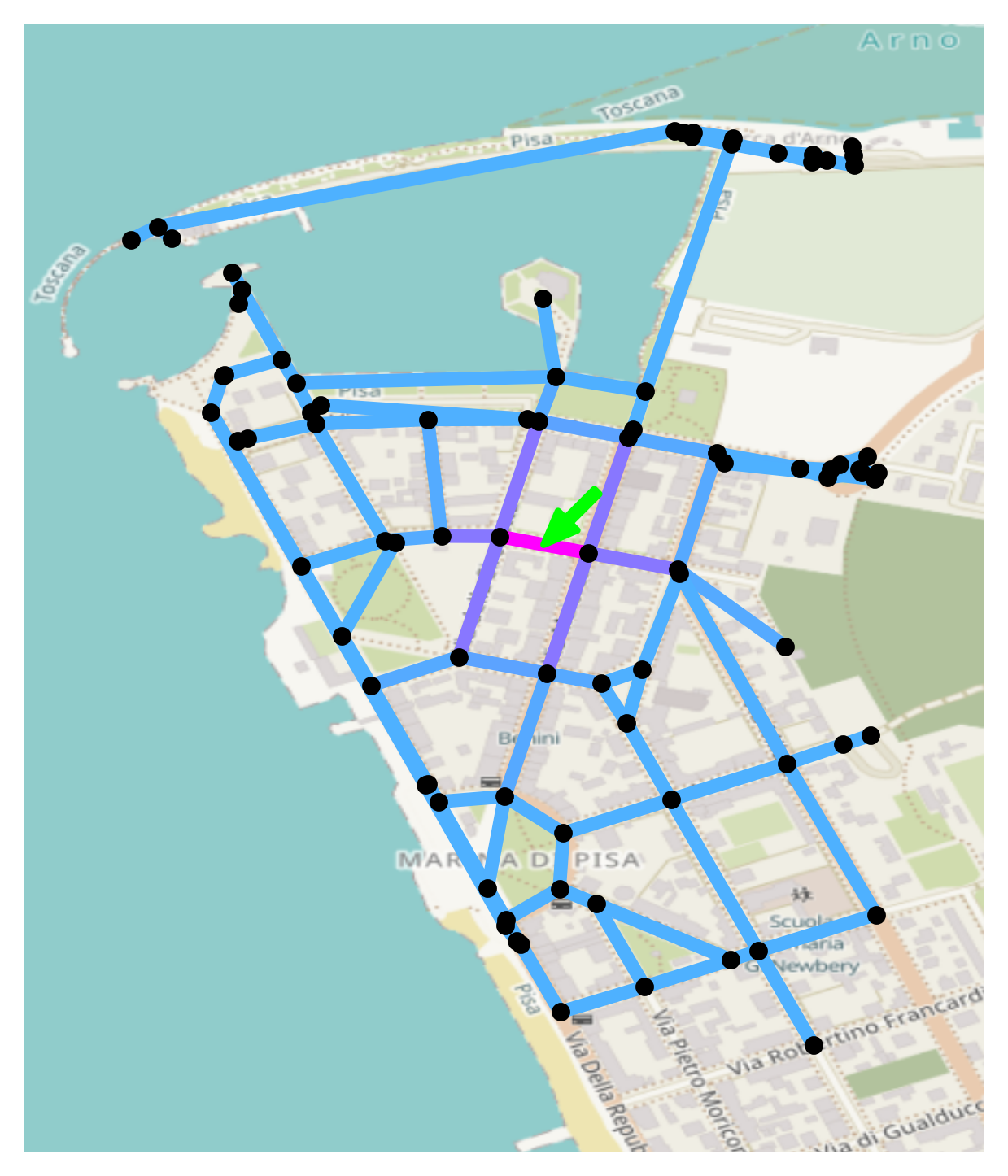}} \\
	\end{tabular}
	\caption{\rev{Marginal edge} Katz centralities of the failure-aware temporal multiplex network model in comparison to the standard temporal multiplex network model.
		In the cycling network in the harbour region of Marina di Pisa, Italy, the most central edge \rev{(highlighted by the green arrow)} is downdated from time step 1 to time step 2 and restored from time step 2 to time step 3.
        \rev{Panels (a)-(e) show marginal edge Katz centralities, while panels (f)-(g) show centrality differences on a separate color scale for better visibility.}
		The failure-aware model puts a local emphasis on the up- and downdates.
		In comparison to the standard model, changes in edge centralities are spatially constrained \rev{while panel (f) illustrates that standard temporal Katz centralities vary throughout the whole network several time steps after the last network update.}}\label{fig:illustration_marina}
\end{figure}

%\todo[inline]{VS: plots (a-c-e-g) look identical, and so (b-d-f)}

%%%%%%%%%%%%%%%%%%%%%%%%%%%%%%
\subsection{A failure-aware temporal multiplex network model}\label{sec:FATE}

Edge up- or downdates in the standard temporal multiplex network model may alter centrality scores over large parts of the networks.
\rev{\Cref{fig:illustration_marina_diff_std} shows that even many time steps after the last update, temporal edge Katz centralities change throughout the whole network in this case.}
%the downdate of the most central road in a cycling network leaves the marginal node centrality of the downdated \vs{road path}
%pipe approximately unchanged while centrality values of edges throughout the whole network increase.

We propose a failure-aware temporal multiplex network model that puts a localized emphasis on the affected network region.
This is beneficial to, e.g., monitor disruptions in infrastructure networks by promoting centrality changes in the affected edges and their neighborhood.
This behavior is illustrated in \Cref{fig:illustration_marina}.
In particular, downdating the most central edge of the network after time step 1 reduces the edge's marginal node centrality under the failure-aware temporal multiplex network model to the minimal value of $1$ in time step 2.
Albeit the recovery of the edge after time step 2, it takes several additional time steps without up- or downdates to restore the edge's original marginal edge centrality.
The model hence retains comparability of centrality scores across time layers while reflecting disruptions in a clear and lasting manner.

To achieve this behavior, we consider the difference between consecutive time layers via the up-/downdate matrices $\bm{W}_\ell=\bm{A}_{\ell+1}-\bm{A}_{\ell}, \ell=1,\dots,L-1$.
We assume the number of up- or downdates from one time step to the other to be moderate such that $\bm{W}_\ell$ is of low-rank.
For example, the up- or downdate of a single edge may be written as the rank-2 matrix $\bm{W}_\ell=\pm (\bm{e}_j\bm{a}_j^\top+\bm{a}_j\bm{e}_j^\top)$, where we define $\bm{a}_j=\bm{A}_\ell\bm{e}_j$.
We propose to replace the identity matrices as temporal inter-layer coupling matrices in the standard model \eqref{eq:block_matrix_standard} by the low-rank matrices $\bm{W}_\ell$.
Hence, the supra-adjacency matrix representation of the failure-aware temporal multiplex network model reads

% To better account for the information transfer between time layers, we introduce
% a \textcolor{red}{\textit{transfer matrix}} encoding the changes between successive time instances of the network, that defines a new multiplex matrix $\bm {\mathcal{M}}$. As temporal networks in the real-world often change connectivity,  be it via breakages in the water network or in deactivation of neurons in the brain, we want to incorporate these temporal changes into the model. We consider the difference between two states of the network to be given by $\bm{W}_\ell=\bm{A}_{\ell+1}-\bm{A}_{\ell}.$ We assume that unless catastrophic changes happen within the (water) network, the modification to the network via $\bm{W}_\ell$ between two time-points is of low-rank. This means that pipe breakages lead to negative and pipe repairs to positive entries in $\bm{W}_\ell$. If $\bm{W}_\ell$ represents the repair or breakage of a single pipe, it is rank two and explicitly given by $\bm{W}_\ell=\pm (\bm{e}_j\bm{a}_j^T+\bm{a}_j\bm{e}_j^T)$, respectively, where $\bm{a}_j=\bm{A}_\ell\bm{e}_j$.  With this transfer matrix, the network over a certain time-range can be modeled as the supra-adjacency matrix $\bm{\mathcal{M}}_L$ of the failure-aware temporal multiplex network
%
\begin{equation}
\label{eq:block_matrix}
\bm{\mathcal{M}}_L=
\begin{bmatrix}
\bm{A}_1&\bm{W}_1&\bm{0}&\dots&\bm{0}\\
\bm{0}&\bm{A}_2&\bm{W}_2&\dots&\bm{0} \\
\bm{0}&\bm{0}&\bm {A}_3 & \ddots&\vdots \\
\vdots & \vdots&\dots & \ddots&\bm{W}_{L-1} \\
\bm{0} & \bm{0} &\dots&\dots& \bm{ A}_L
\end{bmatrix}
\in \mathbb{R}^{NL\times NL}.
\end{equation}%\label{eq:failure}
We study this network representation in the remainder of this paper with the goal to exploit the structure of the matrix $\bm{\mathcal{M}}_L$ for numerical computations.

%%%%%%%%%%%%%%%%%%%%%%%%%%%%%%%
\section{Edge-based temporal Katz centrality}
Centrality measures are a crucial tool in network analysis to study how networks respond to network failures, cf.~\cite{das2018study,rodrigues2018network} and the references therein.
We introduce Katz centrality, which is well-established on single-layer networks and has recently been studied on multiplex networks \cite{bergermann2022fast,bergermann2021orientations,bergermann2023twitter,el2024tensor}.
%Centrality measures have been studied in a variety of contexts for standard graphs (cf. \cite{rodrigues2018network}) and for multilayer networks (cf. \cite{taylor2021tunable,bergermann2022fast}).

%%%%%%%%%%%%%%%%%%%%%%%%%%%%%%%
\subsection{Katz centrality for single-layer network}

Katz centrality is one member of the family of matrix function-based centrality measures \cite{katz1953new}. In the notation of \Cref{sec:multiplex_models}, the Katz centrality vector $\mathbf{x} \in \mathbb{R}^N$ is defined as
\[
\mathbf{x} = \bm A \mathbf{x} + \mathbf{1},
\]
where $\bm{A}\in\R^{N \times N}$ denotes the edge adjacency matrix.
Remember that the usual damping parameter $\alpha$ is absorbed in the definition of $\bm{A}=\alpha\widetilde{\bm{A}}$ for ease of notation, cf.~\Cref{sec:standard_model}.
Solving for the vector $\bm{x}$, we may define Katz centrality as row sums of the resolvent function
\[
\mathbf{x} = (\bm I - \bm A)^{-1} \mathbf{1}.
\]
It is a well-known fact from graph theory that the quantity $[\bm{A}^p]_{ij}$ for a binary adjacency matrix $\bm{A}$ records the number of distinct walks of length $p$ from node $i$ to node $j$.
Using the Neumann series of the resolvent function
\[
(\bm{I} - \bm{A})^{-1} = \sum_{p=0}^{\infty} \bm{A}^p
\]
hence allows us to interpret
\[
\mathbf{x} = \sum_{p=0}^{\infty} \bm{A}^p \mathbf{1}.
\]
as a weighted sum over all walks of all lengths in the network.
One remarkable feature of Katz centrality is that the choice of the parameter $\alpha$ allows to interpolate between local degree and global eigenvector centrality \cite{benzi2015limiting}.

% We can see from this that the contribution of a node to the centrality of another node is determined by the number of walks of all lengths connecting them, with longer walks weighted by the factor $\alpha^k$. We now carry this over to the multilayer water network.
% $(\bm{I}-\bm{A}_1)^{-1}\bm{1}$

%%%%%%%%%%%%%%%%%%%%%%%%%%%%%%%
\subsection{Katz centrality for failure-aware temporal multiplex networks}\label{sec:def_FATE}
%The multiplex network structure introduced in \eqref{eq:failure} is meant to strengthen the impact edges modifications in time, by keeping track of this network change in the upper diagonal blocks $\bm{W}_\ell, \ell=1,\dots,L-1$.
Katz centrality requires the solution of a linear system involving the shifted adjacency matrix.
For the case of multiplex networks, we define the Katz centrality matrix $\bm{C}=[\bm{c}_1,\dots,\bm{c}_L]\in\mathbb{R}^{N\times L}$ that is obtained by collecting the
$L$ blocks of the solution ${\underline{\bm c}}$ of the linear system
\begin{equation}\label{centrality}
(\bm{I} - \bm{\mathcal{M}}_L) {\underline{\bm c}}= \bm{1}, \qquad with \quad {\underline{\bm c}}=
\begin{bmatrix} \bm{c}_1\\\vdots\\\bm{c}_L\end{bmatrix}, \, \,\bm{1}\in\mathbb{R}^{NL}. 
\end{equation}
%
%where $\bm{1}\in\mathbb{R}^{NL}$ denotes the vector of all ones. 
The centrality of node $i$ at time $\ell$ is the entry $(i,\ell)$ of the matrix $\bm{C}$. We thus have a node- and layer-dependent Katz centrality.  Each layer represents the configuration of the spatial network at one point in time.
Ameding standard terminology \cite{taylor2017eigenvector,bergermann2022fast} to the case of edge networks, we call this centrality measure temporal Katz \emph{edge} centrality (marginal edge centrality) of \rev{edge} $i$:
 \begin{equation}\label{eq:def_kmn}
 {\KMN}(i) = \frac{1}{L}\bm{e}_i^T\bm{C}{\bm 1}_L, \qquad i=1,\dots,N. %, \qquad {\bm 1}_L\in\R^L.
 \end{equation}
 %where ${\bm 1}_L=\left[1,\ldots,1\right]^\top\in\R^L.$
%
The quantity ${\KMN}(i)$ describes the centrality of node $i$ over the temporal evolution of the multiplex network.

 We adopt standard terminology \cite{taylor2017eigenvector,bergermann2022fast} to denote temporal Katz \emph{layer} centrality (marginal layer centrality) at time point $\ell$ as
 \begin{equation}\label{eq:def_kml}
 \KML(\ell) = \bm{1}_n^\top\bm{C}\bm{e}_\ell, \qquad \ell=1,\dots,L.
 \end{equation}
 The quantity $\KML(\ell)$ may be of interest as a scalar measure of the overall network communicability at every time snapshot, i.e., reflecting the impact of up- or downdates.

%%%%%%%%%%%%%%%%%%%%%%%%%%%%%%%%%%%
\section{Structural properties of the multiplex matrix inverse}\label{sec:structural_analysis}
We now discuss some properties of the inverse of the block matrix $(\bm{I} - \bm{\mathcal{M}}_L)$ that will be convenient for making computations for large-scale networks efficient.
For notational convenience, we denote the matrix resolvent function of a given matrix $\bm A$ by
\begin{equation}\label{eqn:resolvent}
    \K(\bm{A}) := (\bm{I} - \bm{A})^{-1}.
\end{equation}
We start with the following proposition, which is a generalization of \cite[Theorem 3.1.]{fenu2017block}\vs{, where it was stated for $W_i=I$ for all $i$'s}.
%\todo[inline]{VS: to be checked}

\begin{proposition}\label{prop:block_inverse}
Let $\bm{X}_{k,\ell}\in{\mathbb R}^{N \times N}$ be the $(k,\ell)$ block of $\K(\bm{\mathcal{M}}_L)$. Then for $k\leq \ell$, $\ell = 1, \dots, L$, it holds that
\begin{equation}
\label{eq:form_upper_triangular_blocks}
\bm{X}_{k,\ell} = 
\left( \prod_{i=k}^{\ell-1} \K(\bm{A}_i ) \bm{W}_i \right) \K(\bm{A}_\ell ).
\end{equation}
%while the other blocks are zero
\end{proposition}

The proof of this result is given in Appendix~\ref{proof:block_inverse}.
Throughout the paper we make use of the following fact.
\begin{lemma}\label{lem:inv_prod}
For nonsingular matrices $\bm{I} - \bm{A}$ and $\bm{I} - (\bm{A} - \bm{W})$, 
%the following relation holds
%\begin{equation}\label{eq:product_relation}
\vs{it holds that $\K(\bm{A}) \bm{W} \K(\bm{A} + \bm{W}) =
\K(\bm{A} + \bm{W}) - \K(\bm{A})$}.
%,
%K(\bm{A}) \bm{W} K(\bm{A} - \bm{W}) = K(\bm{A} - \bm{W}) - K(\bm{A}).
%\end{equation}
\end{lemma}
%\begin{proof} 
\vs{{\it Proof.} Recalling the definition of $\K(\cdot)$ in
(\ref{eqn:resolvent}) we have
%\todo[inline]{KB @ VS: I changed to K-notation without touching the proof, which I believe you wanted to shorten.}
%We have
\begin{eqnarray*}
(\bm{I} - \bm{A} - \bm{W})^{-1} - (\bm{I} - \bm{A})^{-1} & =&
(\bm{I} - \bm{A})^{-1} \big[ (\bm{I} - \bm{A}) - (\bm{I} - \bm{A} - \bm{W}) \big] (\bm{I} - \bm{A} - \bm{W})^{-1}\\
&=&
(\bm{I} - \bm{A})^{-1} \bm{W} (\bm{I} - \bm{A} - \bm{W})^{-1}.
\qquad \square
\end{eqnarray*}
%\begin{align*}
%& (\bm{I} - \bm{A})^{-1} \bm{W} (\bm{I} - \bm{A} - \bm{W})^{-1} \\
%&= (\bm{I} - \bm{A})^{-1} \big[ (\bm{I} - \bm{A}) - (\bm{I} - \bm{A} - \bm{W}) \big] (\bm{I} - \bm{A} - \bm{W})^{-1} \\
%&= (\bm{I} - \bm{A})^{-1} (\bm{I} - \bm{A}) (\bm{I} - \bm{A} - \bm{W})^{-1} - (\bm{I} - \bm{A})^{-1} (\bm{I} - \bm{A} - \bm{W}) (\bm{I} - \bm{A} - \bm{W})^{-1} \\
%&= \bm{I} (\bm{I} - \bm{A} - \bm{W})^{-1} - (\bm{I} - \bm{A})^{-1} \bm{I} 
%= (\bm{I} - \bm{A} - \bm{W})^{-1} - (\bm{I} - \bm{A})^{-1}.
%\end{align*}
}
%\end{proof}

% Identity matrix plays no role
% (The identity matrix plays no role, the relation also holds for $\mathcal{K}(A) = A^{-1}$).
% \begin{proof}
% Collecting from the left and right of \eqref{eq:product_relation} the two inverses, we have
% \begin{align*}
% K(\bm{A} - \bm{W}) - K(\bm{A}) & = (\bm{I} - (\bm{A} - \bm{W}))^{-1} - (\bm{I} - \bm{A})^{-1}\\
% & = (\bm{I} - \bm{A})^{-1} \left( \bm{I} - \bm{A} - (\bm{I} - (\bm{A} - \bm{W})) \right) (\bm{I} - (\bm{A} - \bm{W}))^{-1}\\
% & = (\bm{I} - \bm{A})^{-1} \bm{W} (\bm{I} - (\bm{A} - \bm{W}))^{-1}.
% \end{align*}
% \end{proof}

%Recalling that 
\vs{Since} $\bm{A}_{\ell}+\bm{W}_\ell=\bm{A}_{\ell+1}$, we can write
\begin{equation}\label{eq:resolvent_product}
\K(\bm{A}_\ell)\bm{W}_\ell\K(\bm{A}_{\ell+1})=\K(\bm{A}_{\ell+1})-\K(\bm{A}_\ell).
\end{equation}
for any $\ell=1, \ldots, L-1$.
Repeated application to \eqref{eq:form_upper_triangular_blocks}, for $k<\ell\le L$, leads to the following expression for the blocks of the matrix $\K(\bm{\mathcal{M}}_L)$:
{\small 
\begin{equation*}
\bm{X}_{k,\ell}=\begin{cases} (\K(\bm{A}_{k+1})-\K(\bm{A}_{k}))\displaystyle{\prod_{i=1}^{\frac{\ell-k-1}2}\bm{W}_{k+2i-1}}(\K(\bm{A}_{k+2i+1})-\K(\bm{A}_{k+2i}))
& \text{$\ell-k$ odd,}\\

\K(\bm{A}_k)\bm{W}_k\bm{X}_{k+1,\ell}
&\text{$\ell-k$ even.} 
\end{cases}
\end{equation*}
}
Note that if some $\bm{W}_\ell$ are zero, then \vs{$\bm{\mathcal{M}}_L$ becomes block diagonal with variable block sizes, so that} certain blocks above the block super-diagonal of $\K(\bm{\mathcal{M}}_L)$ are also zero. 
%Before we discuss further properties we look at a small example.

\begin{example}\label{ex:four_block_system}
As a small example of this structure, consider the matrix
\begin{equation}\label{eq:4x4_example}
\bm{\mathcal{M}}_4 =
\begin{bmatrix}
\bm{A}_1&\bm{W}_1&\bm{0}&\bm{0}\\
\bm{0}&\bm{A}_2&\bm{W}_2&\bm{0} \\
\bm{0}&\bm{0}&\bm{A}_3 & \bm{W}_3\\
\bm{0}&\bm{0}&\bm{0}&\bm{A}_4
\end{bmatrix}.
\end{equation}
Then 
{\scriptsize
\begin{eqnarray*}
&&\K(\bm{\mathcal{M}}_4)=\\
&& \begin{bmatrix}
\K(\bm{A}_1)&\K(\bm{A}_2)-\K(\bm{A}_1)&\K(\bm{A}_1)\bm{W}_1(\K(\bm{A}_3)-\K(\bm{A}_2))&(\K(\bm{A}_2)-\K(\bm{A}_1))\bm{W}_2(\K(\bm{A}_4)-\K(\bm{A}_3))\\
\bm{0}&\K(\bm{A}_2)&\K(\bm{A}_3)-\K(\bm{A}_2)&\K(\bm{A}_2)\bm{W}_2(\K(\bm{A}_4)-\K(\bm{A}_3)) \\
\bm{0}&\bm{0}&\K(\bm{A}_3) & \K(\bm{A}_4)-\K(\bm{A}_3)\\
\bm{0}&\bm{0}&\bm{0}&\K(\bm{A}_4)
\end{bmatrix}.
\end{eqnarray*}
}
\end{example}

%For nonzero $\bm{W}_\ell$, we analyze the contribution of $\bm{X}_{\ell,.}$ to $K(\bm{\mathcal{M}}_L)\bm{1}$ in the following.

In the following, we analyze the entry decay of the blocks $\bm{X}_{k,\ell}$. These estimates will allow us to derive a suitable truncation strategy of the inverse to make computations more efficient.
%we bound the entries $\bm{X}_{k,\ell}\bm{1}$ with $\ell>k+1$ and $\bm{X}_{k,\ell}$ defined in \eqref{eq:form_upper_triangular_blocks} in order to motivate a truncation strategy discarding any contributions of $(\bm{I}-\bm{\mathcal{M}}_L)^{-1}\bm{1}$ above the first super-diagonal. Given the block matrix \ref{eq:block_matrix},
For the remainder of the section, we let $\bm{W}_\ell=\pm(\bm{e}_{\il}\bm{a}_{\il}^T+\bm{a}_{\il}\bm{e}_{\il}^T)$, where $\il$ is the index of the up- or downdated edge at time $\ell$. %In the following, we use the short-hand notation $i$ to indicate the graph node/vertex $v_i$, so that the shortest distance between the nodes $v_i$ and $v_j$ is denoted by $d(i,j)$.

%in this case $\bm{W}_\ell=\bm{A}_{\ell+1}-\bm{A}_\ell$  be  nodes deleted or recovered from the network represented by the adjacency matrix $\bm{A}_1$, whose indexes are $\im$, for $\ell=1,\dots,L$ 

% {\color{red}and  which are respectevely at shortest path distance from each other $d(\il ,\illl)$.}
%Since edge can be deleted or restored, $\bm{W}_\ell=\pm(\bm{e}_{\il}\bm{a}_{\il}^T+\bm{a}_{\il}\bm{e}_{\il}^T)$, where $\bm{a}_{\il}=\bm{A}_{\ell}\bm{e}_{\il}$.
%$\bm{W}_2=\pm(\bm{e}_{\ii}\bm{a}_{\ii}^T+\bm{a}_{\ii}\bm{e}_{\ii}^T)$, where $\bm{a}_{\ii}=\bm{A}_2\bm{e}_{\ii}$.  Let $\bm{A}_2$ be the adjacency matrix of the network after the node $\i$ has been removed.

We will denote with $s(i,j)$ the shortest path\footnote{A path between any two nodes of index $s_1$ and $s_2$ is a sequence of edges $p(s_1,s_2)=\{(u_1,v_1),\dots,(u_k,v_k)\}$ such that $u_1=s_1$, $v_k=s_2$, $v_1\neq v_2\neq \dots \neq v_k$ and $v_i=u_{i+1}$. The length of a path is the number of edges in a path, and we will denote it as $|(p(s,k))|$.} between the nodes $i$ and $j$, and with $d(i,j)$ its length, which is thus the distance between $i$ and $j$. %  $|s(i,j)|$.

\begin{lemma}\label{lemma:lemma_1}
Let $d(i, j)$ be the shortest path distance between nodes $i$ and $j$. For any neighbor $i_s$ of $i$, it holds that
\begin{equation}
|d(i, j) - d(i_s, j)| \leq 1.
\end{equation}
% Let $g_i$ be the degree of node $i$ and let $d(i,j)$ be the shortest path length between $i$ and a given $j$. Let $i_s$, $s=1,\dots,g_i$ be any node adjacent to the node $i$. Then 
% \begin{equation*}
% d(i,j)-1\leq d(i_s,j)\leq d(i,j)+1.
% \end{equation*}
\end{lemma}
% \todo[inline]{VS: is this proof needed? any classical reference?}
\begin{proof}
Since $i$ and $i_s$ are adjacent, $d(i, i_s) = 1$. The result follows directly from the triangle inequality for graph metrics: $d(i_s, j) \leq d(i_s, i) + d(i, j)$ and $d(i, j) \leq d(i, i_s) + d(i_s, j)$.
\end{proof}

% Let $s(i_s,j)$ be any shortest path between $i_s$ and $j$. It cannot be that $d(i_s,j)=|s(i_s,j)|<d(i,j)-1$. Suppose it was the case. We can always consider the path going from $i$ to $i_s$ composed with $s(i_s,j)$. It's length is $|(i,i_s)\cup s(i_s,j)|=1+|s(i_s,j)|\leq d(i,j)-1< d(i,j)$ which is absurd. Instead, suppose $d(i_s,j)>d(i,j)+1$ and let $s(i,j)$ be any shortest path going from  $i$ to $j$. We can always consider the path going from $i_s$ to $i$ composed with the path $s(i,j)$. It.s length is $|(i_s,i)\cup s(i,j)|=d(i,j)+1$ which is absurd.
%\end{proof}

We recall the following result from \cite{demko1984decay}.
\begin{lemma}\label{lem:lemma2} 
Let $\bm{H}$ be a symmetric positive definite matrix with spectrum contained in $[a, b]$, and let \vs{its} condition number be $\kappa = b/a$. Let $d(i, j)$ denote the shortest path distance between nodes $i$ and $j$ in the graph representation of $\bm{H}$. Then, the entries of the inverse satisfy:
\begin{equation}
\label{eq:bound_entries_of_inverse}
| [\bm{H}^{-1}]_{ij} | \leq C \lambda^{d(i,j)},
\end{equation}
where $\lambda = \frac{\sqrt{\kappa}-1}{\sqrt{\kappa}+1}$ and $C=\frac{(1 + \sqrt{\frac{b}{a}})^2}{2b}$.
\end{lemma}

\begin{remark}
This result is classical in the analysis of sparse matrix inverses. The exponential decay rate $\lambda$ depends solely on the condition number $\kappa$.
\end{remark}
% \begin{lemma}\label{lem:lemma2} 
% Let $\bm{H}$ be a symmetric positive definite matrix with $[a,b]$ being the smallest interval containing its spectrum, and let
% $\kappa=b/a$. % $\sigma(\bm{A})$,.
% Further, let $d(i,j)$ denote the shortest path distance between the nodes $i$ and $j$ in the graph representation of $\bm{H}$. Then, it holds that
% \begin{equation}\label{eq:bound_entries_of_inverse}
% |[\bm{H}^{-1}]_{ij}|\le C \lambda^{d(i,j)}.
% \end{equation}
% with $0<\lambda=\frac{\sqrt{\kappa}-1}{\sqrt{\kappa}+1}$ and $C=\max \{\frac{1}{a},\frac{(1+\sqrt{\kappa})^2}{2b}\}$.
% \end{lemma}

We apply \eqref{eq:bound_entries_of_inverse} to the diagonal blocks $\K(\bm{A}_\ell)$ of $\K(\bm{\mathcal{M}}_L)$.
Thanks to the scaling in (\ref{eq:def_alpha}), 
%in which case the following worst-case bound on the value of $\lambda$ in \eqref{eq:bound_entries_of_inverse} holds. Each $\bm{A}_\ell$ is such that their spectra 
the spectrum of $\bm{A}_\ell$ is contained in $(-c_\alpha,+c_\alpha]$ for some $0<c_\alpha<1$, so that the spectrum of each $\bm{I}-\bm{A}_\ell$ is contained in $[1-c_\alpha,1+c_\alpha)$.  
    Then, \eqref{eq:bound_entries_of_inverse} holds for $\bm{H}=\bm{I}-\bm{A}_\ell$, $\ell=1,\dots,L$, where $\lambda$ satisfies
    \begin{equation}
        \lambda\leq \frac{\sqrt{\frac{1+c_\alpha}{1-c_\alpha}}-1}{\sqrt{\frac{1+c_\alpha}{1-c_\alpha}}+1}. \qquad 
    \end{equation}

%\begin{proof}
   % By construction, the matrices $\bm{A}_\ell$ are symmetric, indefinite, non-negative, and irreducible and their spectra are contained in the interval $(-c_\alpha,c_\alpha]$.
   % Consequently, the spectra of $(\bm{I}-\bm{A}_\ell)$ are contained in the interval $[1-c_\alpha,1+c_\alpha)$, i.e., in the notation of \Cref{lemma:lemma_1}, we have $a=1-c_\alpha$ and $b\leq 1+c_\alpha$ and hence $\frac{b}{a}\leq \frac{1+c_\alpha}{1-c_\alpha}>1$.
    %Inserting this into the definition of $\lambda$ in \eqref{eq:bound_entries_of_inverse} completes the proof.
   % \end{proof}

For $c_\alpha\rightarrow 1$ it holds that  $\lambda\rightarrow 1$. In this case, 
the  Katz centrality converges to the
eigenvector centrality \cite{benzi2015limiting}.
%However, Benzi and Klymko showed that Katz centrality converges to eigenvector centrality in the limit $c_\alpha\rightarrow 1$, where the latter is easily computed by a power method.
For some typical parameter choices such as $c_\alpha=0.5, 0.8$, and $c_\alpha=0.9$, we obtain quite favorable values of the upper bound, that is
$0.268, 0.5$ and $0.627$, respectively.
%$\lambda\leq 0.268$, $\lambda\leq 0.5$, and $\lambda\leq 0.627$, respectively.
For $c_\alpha=0.99$  it holds that $\lambda\leq 0.868$.

%In a first step, we consider the example network from \eqref{eq:matpert} and bound the quantity $\bm{X}_{1,3}\bm{1}$.

The following bound is the main result of this section. We show that $\|{\bm X}_{k,\ell}\textbf{1}\|_2$ decays in a way that depends on the distance between the involved nodes.
This forms the basis for the truncation strategies proposed in \Cref{sec:computations} that omit negligible contributions $\|{\bm X}_{k,\ell}\textbf{1}\|_2$ to linear systems involving $(\bm{I} - \bm{\mathcal{M}}_L)$ or $(\bm{I} - \bm{\mathcal{M}}_{L+1})$.

\begin{proposition}\label{prop:block_decay} 
Let $\bm{X}_{k,\ell}$ be as in (\ref{eq:form_upper_triangular_blocks}) and let $[a,b]$ be the smallest interval containing the spectrum of $\bm{A}_m$ for every $m=k,\dots,\ell$. Let $[\bm{c}]_\ell$ be the maximum entry of the Katz centrality vector of ${\bm A}_\ell$ and $g_{\max}$ be the maximum degree of the network. \vs{Finally, let $i_m$ be the index of the removed/restored edge at time $m$.} Then

\begin{equation}\label{eq:bound_entries}
||{\bm X}_{k,\ell}\textbf{1}||_2\le \bar{C}\lambda^{\sum_{m=k}^{\ell-2}(d(\im,\imm)-1)},
\end{equation}
with $\lambda=\frac{\sqrt{\frac{1+c_\alpha}{1-c_\alpha}}-1}{\sqrt{\frac{1+c_\alpha}{1-c_\alpha}}+1}$ and {$\bar{C}=\frac{1}{1-c_{\alpha}}(2(1+\alpha g_{\max})C)^{\frac{\ell}{2}}[\bm{c}_\ell]_{\max}$}.
\end{proposition}

\begin{proof}[Sketch of the proof]
The block $\bm{X}_{k,\ell}$ can be expressed as a product of operators involving the resolvents $\K(\bm{A}_m)$ and the disruption matrices ${W}_m$, \vs{$m=k, \ldots, \ell$}, cf. \Cref{prop:block_inverse}. 
Specifically, $\bm{X}_{k,\ell}$ contains terms of the form $\K(\bm{A}_m) \bm{W}_m \K(\bm{A}_{m+1})$. 
Since $\bm{W}_m$ represents a local disruption at node $i_m$, the norm of this product is governed by the decay of the entries of $\K(\bm{A}_m)$ between nodes $i_m$ and $i_{m+1}$. 
Applying \Cref{lem:lemma2} to each factor in the product yields the exponential decay term $\lambda^{d(i_m, i_{m+1}) - 1}$. 
The constant $\bar{C}$ aggregates the norms of the boundary terms and the maximum degree of the network. 
The detailed derivation of the entry-wise bounds and norm chaining is provided in Appendix \ref{app:proof_prop_4.4}.
\end{proof}

The bound in (\ref{prop:block_decay}) shows that as the
inverse blocks move away from the main diagonal block, their magnitude
decreases, therefore, a good approximation to the inverse of 
$\bm I-\bm{\mathcal{M}}$ can be obtained by dropping the blocks that are
farthest from the diagonal ones.  In other words, a full back substitution may
be unnecessary. We will exploit this result in section~\ref{sec:truncation}.

%\begin{remark}
%The bound in (\ref{prop:block_decay}), associates the magnitude of $\|X_{k,\ell}\textbf{1}\|_2$ with the term $\sum_{m=k}^{\ell-1}d(\im,\imm)-1)$.
%{\color{red}As consecutive distances increase between nodes deleted in two consecutive layers $A_\ell$ and $A_{\ell+1}$, $X_{k,\ell}\textit{1}$ will be less significant and therefore negligible. Conversely, removing closer nodes leads to a more significant 
%$X_{k,\ell}\textbf{1}$.
%The accuracy of an approximation that neglects $X_{k,\ell}\textbf{1}$ will therefore depend on the type of nodes removed as well as the structural characteristics of the network on which our model operates as medium shortest path distance or diameter.}
%\end{remark}

\begin{remark}\label{mindist}
    For simplicity, Proposition~\ref{prop:block_decay} is stated for the case in which each $\bm{W}_m, m=k,\dots,\ell$ represents a single up- or downdate. In case of several such events, $\bm{W}_m$ becomes a sum of rank-2 terms. Then, \eqref{eq:bound_entries} contains the product over this sum and the bound is dominated by the pair of consecutive up- or downdates with minimal distance.
\end{remark}

\begin{example}\label{ex:decay}
	We revisit networks with $L=4$ layers as considered in \Cref{ex:four_block_system} and \vs{described} by
	\begin{equation*}
	\bm{\mathcal{M}}_4 =
	\begin{bmatrix}
	\bm{A}_1&\bm{W}_1&\bm{0}&\bm{0}\\
	\bm{0}&\bm{A}_2&\bm{W}_2&\bm{0} \\
	\bm{0}&\bm{0}&\bm{A}_3 & \bm{W}_3\\
	\bm{0}&\bm{0}&\bm{0}&\bm{A}_4
	\end{bmatrix},
	\end{equation*}
	where $\bm{W}_1 = - \bm{e}_{i_1}\bm{a}_{i_1}^\top - \bm{a}_{i_1}\bm{e}_{i_1}^\top$, $\bm{W}_2 = - \bm{e}_{i_2}\bm{a}_{i_2}^\top - \bm{a}_{i_2}\bm{e}_{i_2}^\top$, and $\bm{W}_3 = - \bm{e}_{i_3}\bm{a}_{i_3}^\top - \bm{a}_{i_3}\bm{e}_{i_3}^\top$ represent the removal of edges $i_1, i_2,$ and $i_3$, respectively.
	By \Cref{prop:block_decay} and \eqref{eq:resolvent_product}, we have
	\begin{align}
	\|\bm{X}_{1,3}\bm{1}\|_2 & = \|\K(\bm{A}_1)\bm{W}_1(\K(\bm{A}_3)\bm{1} - \K(\bm{A}_2)\bm{1})\|_2 && \leq \bar{C}\lambda^{d(i_1,i_2)-1},\label{eq:bound_13}\\
	\|\bm{X}_{2,4}\bm{1}\|_2 & = \|\K(\bm{A}_2)\bm{W}_2(\K(\bm{A}_4)\bm{1}-\K(\bm{A}_3)\bm{1})\|_2 && \leq \bar{C}\lambda^{d(i_2,i_3)-1},\label{eq:bound_24}\\
	\|\bm{X}_{1,4}\bm{1}\|_2 & = \|(\K(\bm{A}_2)-\K(\bm{A}_1))\bm{W}_2(\K(\bm{A}_4)\bm{1}-\K(\bm{A}_3)\bm{1})\|_2 && \leq \bar{C}\lambda^{d(i_1,i_2)+d(i_2,i_3)-2}.\label{eq:bound_14}
	\end{align}
	
	We consider the street network of Mittweida with $N=563$ edges, the parameter $c_\alpha=0.8$, and a randomly chosen index $i_1$.
	Experiments not included in the manuscript show that different choices of $i_1$ do not affect the qualitative behavior.
	Selecting the maximal distance $d_{\mathrm{max}}=20$ between consecutively removed nodes, we compute $\|\bm{X}_{1,3}\bm{1}\|_2, \|\bm{X}_{2,4}\bm{1}\|_2,$ and $\|\bm{X}_{1,4}\bm{1}\|_2$ with linear system tolerance $10^{-12}$.
	\Cref{fig:decay_13,fig:decay_24} show $\|\bm{X}_{1,3}\bm{1}\|_2$ and $\|\bm{X}_{2,4}\bm{1}\|_2$ as functions of $d(i_1,i_2)$ and $d(i_2,i_3)$, respectively, alongside their bounds \eqref{eq:bound_13} and \eqref{eq:bound_24}.
	\Cref{fig:decay_14} illustrates the additivity of distances between consecutive up- or downdates in the exponent of $\lambda$ shown in \Cref{prop:block_decay}.
	The bound \eqref{eq:bound_14} of $\|\bm{X}_{1,4}\bm{1}\|_2$ is shown in \Cref{fig:decay_14_diag} as a function of $d(i_1,i_2)+d(i_2,i_3)$ with  $d(i_1,i_2)=d(i_2,i_3)$, which behaves very similarly to those of \Cref{fig:decay_13,fig:decay_24}.
\end{example}

\begin{figure}[t]
	\centering
	\subfloat[$\|\bm{X}_{1,3}\bm{1}\|_2$]{\includegraphics[width=0.48\linewidth]{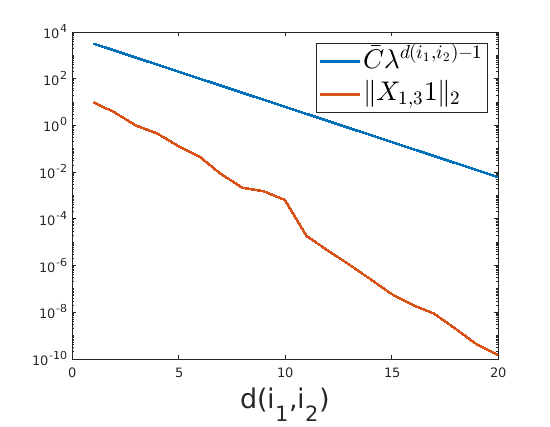}\label{fig:decay_13}}
	\hfill
	\subfloat[$\|\bm{X}_{2,4}\bm{1}\|_2$]{\includegraphics[width=0.48\linewidth]{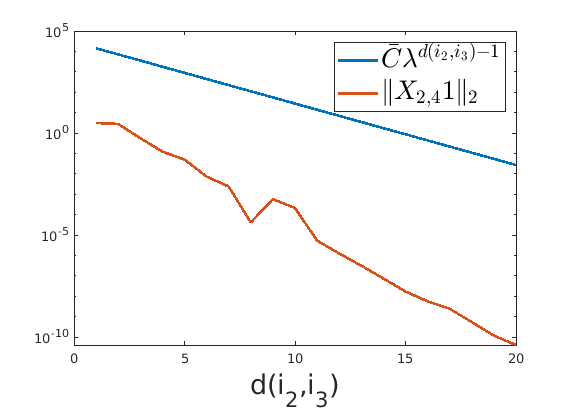}\label{fig:decay_24}}
    
	\subfloat[$\|\bm{X}_{1,4}\bm{1}\|_2$]{\includegraphics[width=0.48\linewidth]{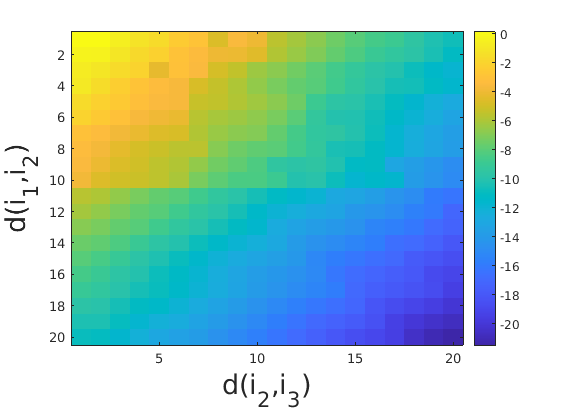}\label{fig:decay_14}}
    \hfill
	\subfloat[$\|\bm{X}_{1,4}\bm{1}\|_2$]{\includegraphics[width=0.48\linewidth]{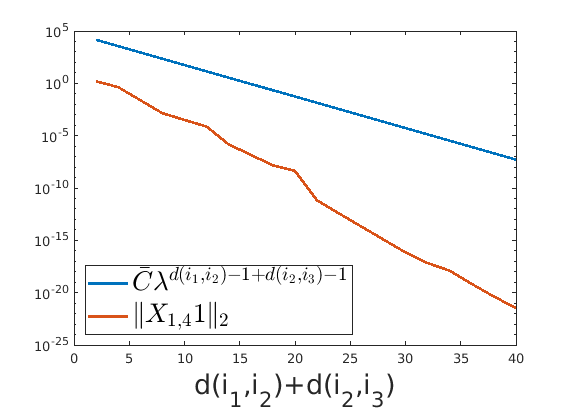}\label{fig:decay_14_diag}}
	\caption{Comparison of $\|\bm{X}_{1,3}\bm{1}\|_2, \|\bm{X}_{2,4}\bm{1}\|_2,$ and $\|\bm{X}_{1,4}\bm{1}\|_2$ with their bounds for the street network of Mittweida.
    Panels (a) and (b) show $\|\bm{X}_{1,3}\bm{1}\|_2$ and $\|\bm{X}_{2,4}\bm{1}\|_2$ as functions of $d(i_1,i_2)$ and $d(i_2,i_3)$, respectively.
	Panel (c) shows $\|\bm{X}_{1,4}\bm{1}\|_2$ as a function of both $d(i_1,i_2)$ and $d(i_2,i_3)$.
	Panel (d) shows the ``diagonal'' of panel (c), i.e., the quantity $\|\bm{X}_{1,4}\bm{1}\|_2$ as a function of $d(i_1,i_2)+d(i_2,i_3)$ with  $d(i_1,i_2)=d(i_2,i_3)$ alongside its bound.}
\end{figure}

\section{Computational strategies}\label{sec:computations}
The main computational cost associated with the indices \KMN\ and \KML\ defined in \eqref{eq:def_kmn} and \eqref{eq:def_kml}, respectively, arises from solving the linear system
\begin{equation}\label{eqn:sys}
(\bm{I} - \bm{\mathcal{M}}_L)\underline{\bm c} = \bm{1},
\end{equation}
where the system matrix is upper block-bidiagonal. As time proceeds, $L$ grows and so does the number of blocks in the matrix.
This raises three practical questions that we address in turn: whether the system can be solved approximately at lower cost without 
sacrificing accuracy (\Cref{sec:truncation}), whether the solution can be updated as $L$ grows without recomputing the whole system from scratch (\Cref{sec:truncation_strategy}), and how well the resulting strategies scale on parallel hardware for large-scale network sizes that one encounters in real infrastructure networks (\Cref{sec:scaling-experiments}). We derive each strategy and immediately assess it numerically on the test networks of \Cref{sec:data}, before closing with the parallel-scaling study.

\subsection{Truncation of the block matrix}\label{sec:truncation}
Solving \eqref{eqn:sys} directly admits two standard strategies: block back-substitution, which is exact but inherently sequential,
 since block $\ell$ depends on block $\ell+1$; or applying an iterative method such as BiCGStab to the full system, which does not exploit the block structure at all.
 The scaling \eqref{eq:def_alpha} clusters the spectrum of $\bm{\mathcal{M}}_L$ conveniently around $1$, leading to rapid convergence of the latter. 
 However, neither strategy takes advantage of the entry decay proven in \Cref{sec:structural_analysis}, %which may lead 
 \vs{possibly leading}
 to unnecessary computations.

 In this section, we propose an approximation of $\K(\bm{\mathcal{M}}_L)$ that \vs{mainly} requires solving the $L$ decoupled diagonal-block systems $\K(\bm{A}_\ell)\bm{1}$, $\ell=1,\dots,L$ of the single-layer network of size $N$. \vs{This computation is}, by construction, embarrassingly parallel.
 It follows from the structure of $\K(\bm{\mathcal{M}}_L)$ derived in \Cref{sec:structural_analysis} that this suffices to represent the block-diagonal as well as the block-superdiagonal of $\K(\bm{\mathcal{M}}_L)$ exactly.
 The rationale behind this approach is that large expected distances between up- or downdated edges in consecutive time steps lead to a rapid entry decay that makes the contribution of $\|\bm{X}_{k,\ell}\bm{1}\|_2$ with $\ell\geq k+1$ negligible.
 Clearly, smaller values of $\|\bm{X}_{k,\ell}\bm{1}\|_2$ make the approximation more accurate.
 \vs{Nonetheless, care should be taken in avoiding a greedy block truncation. Indeed, if only the first upper diagonal block is retained when farther diagonals are away from nonzero, the truncated matrix  $\K(\bm{\mathcal{M}}_L)$ is no more informative than that of a simple block diagonal supra-adjacency matrix. In the following we provide one such example.}

 \begin{example}
 	In the setting of \Cref{ex:four_block_system}, \vs{for $L=4$ assume that truncation to the first upper block is performed, leading 
    to the following approximation of the failure-aware temporal multiplex Katz centrality $\K(\bm{\mathcal{M}}_4)\bm{1}$,
    %so that the truncated $\K(\bm{\mathcal{M}}_4)$ looks as follows,}
    %we define the truncation of $K(\bm{\mathcal{M}}_4)$ above the superdiagonal as
    {\small
 	\begin{equation*}
 	\overline{\K(\bm{\mathcal{M}}_4)}\bm{1} \!\!=\!\!
 	\begin{bmatrix}
 	\K(\bm{A}_1) & \K(\bm{A}_2) - \K(\bm{A}_1) & \bm{0} & \bm{0} \\
 	\bm{0} & \K(\bm{A}_2) & \K(\bm{A}_3) - \K(\bm{A}_2) & \bm{0} \\
 	\bm{0} & \bm{0} & \K(\bm{A}_3) & \K(\bm{A}_4) - \K(\bm{A}_3) \\
 	\bm{0} & \bm{0} & \bm{0} & \K(\bm{A}_4) .
 	\end{bmatrix}\!\bm{1}\! =\!
 \begin{bmatrix}
 \K(\bm{A}_2)\mathbf{1} \\
 \K(\bm{A}_3)\mathbf{1} \\
 \K(\bm{A}_4)\mathbf{1} \\
 \K(\bm{A}_4)\mathbf{1}
 \end{bmatrix}\!.   
 \end{equation*}
 }
 }
 %\vs{This approximation leads to the following approximation of the failure-aware temporal multiplex Katz centrality $\K(\bm{\mathcal{M}}_4)\bm{1}$,}
 %being approximated by
 %$$
 %\overline{\K(\bm{\mathcal{M}}_4)}\bm{1} = \begin{bmatrix}
 %\K(\bm{A}_1)\mathbf{1} + (\K(\bm{A}_2) - \K(\bm{A}_1))\mathbf{1} \\
 %\K(\bm{A}_2)\mathbf{1} + (\K(\bm{A}_3) - \K(\bm{A}_2))\mathbf{1} \\
 %\K(\bm{A}_3)\mathbf{1} + (\K(\bm{A}_4) - \K(\bm{A}_3))\mathbf{1} \\
 %\K(\bm{A}_4)\mathbf{1}
 %\end{bmatrix}
 %=
 %\begin{bmatrix}
 %\K(\bm{A}_2)\mathbf{1} \\
 %\K(\bm{A}_3)\mathbf{1} \\
 %\K(\bm{A}_4)\mathbf{1} \\
 %\K(\bm{A}_4)\mathbf{1}
 %\end{bmatrix}.
 %$$
% The approximation error is given by the quantities $\|\bm{X}_{13}\bm{1}\|_2, \|\bm{X}_{14}\bm{1}\|_2,$ and $\|\bm{X}_{24}\bm{1}\|_2$, for which we prove bounds in \Cref{prop:block_decay}.

 The example generalizes to an arbitrary number of layers in which case we have 
 {\small 
 $$
 \overline{\K(\bm{\mathcal{M}}_L)}\bm{1} = \begin{bmatrix}
 \K(\bm{A}_2)\mathbf{1} \\
 \K(\bm{A}_3)\mathbf{1} \\
 \vdots \\
 \K(\bm{A}_L)\mathbf{1} \\
 \K(\bm{A}_L)\mathbf{1}
 \end{bmatrix}.
 $$
 }

 Here, \vs{truncation makes the effect of the superdiagonal of $\bm{\mathcal{M}}_L$ containing $\bm{W}_1,\dots,$ $\bm{W}_{L-1}$} so small that no additional information is present in $\overline{\K(\bm{\mathcal{M}}_L)}\bm{1}$ in comparison to single-layer Katz centrality applied to the individual time snapshots.
 In fact, in the block formulation, the centrality $\K(\bm{A}_L)\bm{1}$ is duplicated while that of $\K(\bm{A}_1)\bm{1}$ is omitted.
 \vs{On the other hand, if the solution resulting from this drastic truncation is still highly accurate, it means that our {\it full} failure-aware temporal multiplex model does not provide any extra information with respect to the block diagonal model. This should be an indication that the multiplex network contains very few node modifications with time.}
 %,   the case where $\overline{\K(\bm{\mathcal{M}}_L)}\bm{1}$ approximates $\K(\bm{\mathcal{M}}_L)\bm{1}$ to high accuracy represents a limitation of cf.~\Cref{sec:FATE}.
  \end{example}
  
  \Cref{sec:truncation_illustrative} presents one such example while our new model adds new insights for all other numerical examples presented in this manuscript.

\subsection{Updating and truncating the centrality vector}\label{sec:truncation_strategy}
We now consider the situation in which failure-aware temporal Katz centralities \eqref{centrality} have been computed for $L$ time layers and that a new up- or downdate becomes available for layer $L+1$.
We wish to update the temporal centralities as this new layer arrives, without solving the enlarged $N(L+1)\times N(L+1)$ system
\begin{equation}
(\bm{I}-\bm{\mathcal{M}}_{L+1})\underline{\bm{c}}_{\textrm{new}}=\bm{1}_{N(L+1)}
\end{equation}
from scratch.
Note that the matrix $\K(\bm{\mathcal{M}}_{L+1})$ generally contains an additional fully dense 
%additional 
block column in comparison to $\K(\bm{\mathcal{M}}_{L})$.
Hence, the upper $L$ blocks of $\underline{\bm{c}}_{\textrm{new}}$ generally differ from $\underline{\bm{c}}$ defined in \eqref{centrality}.
We write $\bm{C}_{\textrm{new}}=[\bm{C}_{\textrm{upd}},\bm{c}_{L+1}]\in\mathbb{R}^{N\times(L+1)}$, where $\bm{C}_{\textrm{upd}}$ collects the updated centralities at layers $1,\dots,L$ after incorporating the new up- or downdate at time $L+1$, while $\bm{c}_{L+1}$ denotes the centrality vector of the network described by $\bm{A}_{L+1}$. To relate $\bm{C}_{\textrm{new}}$ to the already computed $\bm{C}$, we partition the enlarged system as
\begin{equation}\label{matrix1}
\begin{bmatrix}
\bm{I} - \bm{\mathcal{M}}_{L} & -\bm{E}_{L}\bm{W}_{L} \\
{\bm 0} & \bm{I} - \bm{A}_{L+1}
\end{bmatrix}
\begin{bmatrix}
\underline{\bm{c}}_{\textrm{upd}}\\
\bm{c}_{L+1}
\end{bmatrix}
=
\mathbf{1}_{N(L+1)},
\end{equation}
where $\bm{E}_{L}=[\bm{0},\dots,\bm{0},\bm{I}_N]^\top\in\mathbb{R}^{NL\times N}$ and $\underline{\bm{c}}_{\textrm{upd}}$ stacks the columns of $\bm{C}_{\textrm{upd}}$ one after the other. From \eqref{matrix1} it follows that
\begin{equation}
\begin{cases}
(\bm{I} - \bm{\mathcal{M}}_{L}) \underline{\bm{c}}_{\textrm{upd}} -\bm{E}_{L}\bm{W}_{L} \bm{c}_{L+1}= \bm{1}_{NL} \\
(\bm{I} - \bm{A}_{L+1}) \bm{c}_{L+1} = \bm{1}_{N} ,
\end{cases}
\end{equation}
so that $\bm{c}_{L+1}=(\bm{I}-\bm{A}_{L+1})^{-1}\bm{1}_{N}$.
Defining $\underline{\bm{c}}_L=(\bm{I} - \bm{\mathcal{M}}_{L})^{-1}\bm{1}_{NL}$, we have
\begin{equation}\label{adj}
\underline{\bm{c}}_{\textrm{upd}} = \underline{\bm{c}}_L + (\bm{I}-\bm{\mathcal{M}}_{L})^{-1}(\bm{E}_{L}\bm{W}_L)\bm{c}_{L+1}.
\end{equation}
The update thus requires solving two linear systems: a small one, involving $(\bm{I}-\bm{A}_{L+1})$, for $\bm{c}_{L+1}$, and a large one, involving $(\bm{I}-\bm{\mathcal{M}}_{L})$, for the correction term in \eqref{adj}. The former is negligible; the latter costs exactly as much as the original solve of \eqref{centrality}, and this full cost recurs at every subsequent update. As with the truncated solve of \Cref{sec:truncation}, we show in \Cref{sec:truncation-numerics} that a truncated evaluation of \eqref{adj} yields a computational advantage that becomes more significant as $L$ increases.

\begin{remark}
We can rewrite the update step as
\begin{equation}\label{rem:c_update}
\underline{\bm{c}}_{L+1} =
\begin{bmatrix}
\underline{\bm{c}}_L \\ \bm{0}
\end{bmatrix}
+ \begin{bmatrix}
(\bm{I}-\bm{\mathcal{M}}_{L})^{-1}(\bm{E}_{L}\bm{W}_L)\bm{c}_{L+1}\\
\bm{c}_{L+1}
\end{bmatrix}
= \begin{bmatrix}
\underline{\bm{c}}_L \\ \bm{0}
\end{bmatrix}
+ \underline{\bm u},
\end{equation}
with $\underline{\bm u}=(\bm{I}-\bm{\mathcal{M}}_{L+1})^{-1}\bm{E}_{L+1}\bm{1}_{N}$, which can be written blockwise as
\begin{equation}\label{eqn:u}
\underline{\bm u}
=\begin{bmatrix}
\bm{X}_{1,L+1}\bm{1} \\
\vdots \\\bm{X}_{L,L+1}\bm{1}
\\\bm{X}_{L+1,L+1}\bm{1}
\end{bmatrix}=
\begin{bmatrix}
\K(\bm{A}_1) \bm{W}_1 \cdots\bm{W}_L \K({\bm A}_{L+1})\bm{1} \\
\vdots \\
\\\K(\bm{A}_L) \bm{W}_L \K(\bm{A}_{L+1})\bm{1} \\
{\K({\bm A}_{L+1})}\bm{1}
\end{bmatrix} .
\end{equation}
\end{remark}

%This block-column expression for $\underline{\bm u}$ has exactly the structure analyzed in \Cref{sec:truncation}, here restricted to the last block column of $K(\bm{\mathcal{M}}_{L+1})$.
In \Cref{sec:truncation}, we discussed truncating $\K(\bm{\mathcal{M}}_L)$ above the superdiagonal.
We now generalize this approach for $\K(\bm{\mathcal{M}}_{L+1})$ by retaining the $h+1$ blocks above the block-diagonal and setting the remaining $L-h$ block components to zero.
This yields the truncated update vector
\begin{equation}\label{eqn:overu}
\overline{\underline{\bm u}}=
\overline{\K(\bm{\mathcal{M}}_{L+1})} \bm{E}_{L+1} \bm{1}_{N}
=
\begin{bmatrix}
0 \\
\vdots\\
0\\
\K(\bm{A}_{L-h}) \bm{W}_{L-h}\cdots\bm{W}_{L} \K({\bm A}_{L+1})\bm{1}\\
\vdots \\
\\\K(\bm{A}_L) \bm{W}_L \K(\bm{A}_{L+1})\bm{1} \\
\K({\bm A}_{L+1})\bm{1}
\end{bmatrix}.
\end{equation}
\rev{The truncated centrality block vector then reads
\begin{equation}\label{eq:trunc_centrality}
\overline{\underline{\bm{c}}}_{L+1} = \begin{bmatrix}
\underline{\bm{c}}_L \\ \bm{0}
\end{bmatrix}
+ \overline{\underline{\bm u}}, \qquad\text{or}\qquad \overline{\underline{\bm{c}}}_{L+1} = \begin{bmatrix}
\overline{\underline{\bm{c}}}_L \\ \bm{0}
\end{bmatrix}
+ \overline{\underline{\bm u}},
\end{equation}
depending on whether $\underline{\bm{c}}_L$ has been computed to the employed tolerance or truncation has previously been applied for obtaining $\underline{\bm{c}}_L$.
Note that the truncated quantities $\overline{\K(\bm{\mathcal{M}}_{L+1})}, \overline{\underline{\bm{c}}}_{L+1}$ and $\overline{\underline{\bm u}}$ depend on the truncation parameter $h$. We suppress this dependence in the notation for readability and instead report the specific choices of $h$ used in \Cref{sec:numerics}.}
We next exploit the structure of $\bm{W}_\ell$ to evaluate the retained blocks by backward recursion, rather than by forming any $\K(\bm{A}_\ell)$ explicitly. After computing $\bm{u}_{L+1}=\bm{c}_{L+1}=\K(\bm{A}_{L+1})\bm{1}$, the next to last block follows from \Cref{lem:inv_prod} since $\bm{A}_{L+1}=\bm{A}_L+\bm{W}_L$ as
$$
\bm{u}_{L}=\K(\bm{A}_L) \bm{W}_L \K(\bm{A}_{L+1})\bm{1}=\K(\bm{A}_{L+1})\bm{1}-\K(\bm{A}_L)\bm{1},
$$
and analogously for the next block,
$$
\bm{u}_{L-1}=\K(\bm{A}_{L-1})\bm{W}_{L-1}\K(\bm{A}_{L})\bm{W}_{L}\K(\bm{A}_{L+1})\bm{1}
=
\K(\bm{A}_{L-1})\bm{W}_{L-1} \bm{u}_{L}.
$$
The remaining terms follow by the same recursion, so that evaluating each of the $h$ retained blocks incurs the cost of a single small linear system solve.
With this strategy, the cost of solving the large $NL\times NL$ system \eqref{adj} is thus reduced to the solution of $h$ small systems of size $N\times N$.

The update procedure without and with truncation is summarized in \Cref{alg:WO_trunc,alg:W_trunc}, respectively.

\begin{algorithm}[t]
	\begin{tabular}{lll}
		Input:
		& $\underline{\bm{c}}_L=(\bm{I} - \bm{\mathcal{M}}_{L})^{-1}\bm{1}_{NL},$ & Katz centralities for times 1 to $L$.\\
		& $\bm{\mathcal{M}}_{L},$ & Full multiplex matrix at time $L$.\\
		& $\bm{A}_{L+1},$ & Next time step's adjacency matrix.\\
		& $\bm{W}_L,$ & Next modification $\bm{W}_L=\bm{A}_{L+1}-\bm{A}_L$.\\
		Output: & $\underline{\bm{c}}_{L+1}$, & Katz centralities for times 1 to $L+1$.
	\end{tabular}
	\vspace{.3em}
	\begin{algorithmic}[1]
		\State Solve $(\bm{I} - \bm{A}_{L+1}) \bm{c}_{L+1} = \bm{1}_{N}$\Comment{size $N\times N$}
		\State Solve $(\bm{I}-\bm{\mathcal{M}}_{L}) \bm{x} = (\bm{E}_{L}\bm{W}_L)\bm{c}_{L+1}$\Comment{size $NL\times NL$}
		\State Set $\underline{\bm{u}} = \begin{bmatrix}\bm{x}\\ \bm{c}_{L+1}\end{bmatrix}$
		\State Set $\underline{\bm{c}}_{L+1} = \begin{bmatrix}\underline{\bm{c}}_L\\\bm{0}\end{bmatrix}+\underline{\bm{u}}$
	\end{algorithmic}
	\caption{Update \emph{without} truncation.}\label{alg:WO_trunc}
\end{algorithm}

\begin{algorithm}[t]
	\begin{tabular}{lll}
		Input:
		& \rev{$\overline{\underline{\bm{c}}}_L,$} & \rev{(Truncated) Katz centralities for times $1$ to $L$.}\\
		& $\bm{A}_{L-h},\dots,\bm{A}_{L+1},$ & Time step adjacency matrices.\\
		& $\bm{W}_{L-h},\dots,\bm{W}_{L-1},$ & Previous modifications.\\
		& $h\in\mathbb{N},$ & Truncation length.\\
		Output: & $\overline{\underline{\bm{c}}}_{L+1}$, & \emph{Truncated} Katz centralities for times 1 to $L+1$.
	\end{tabular}
	\vspace{.3em}
	\begin{algorithmic}[1]
		\State Solve $(\bm{I} - \bm{A}_{L+1}) \bm{c}_{L+1} = \bm{1}_{N}$\Comment{size $N\times N$}
		\State Set $\bm{u}_{L+1}=\bm{c}_{L+1}$
		\State \rev{Extract $\overline{\bm{c}}_L = \bm{E}_L^\top\overline{\underline{\bm{c}}}_L$}
		\State \rev{Set $\bm{u}_{L}=\bm{c}_{L+1} - \overline{\bm{c}}_L$}
		\For{$i=1,\dots,h$}
		\State Solve $(\bm{I} - \bm{A}_{L-i}) \bm{u}_{L-i} = \bm{W}_{L-i}\bm{u}_{L-i+1}$\Comment{size $N\times N$, $h$ times}
		\EndFor
		\State Set $\overline{\underline{\bm u}} =
		\begin{bmatrix}\bm{0}^\top & \cdots & \bm{0}^\top & {\bm u}_{L-h}^\top & \cdots & {\bm u}_{L}^\top & {\bm u}_{L+1}^\top\end{bmatrix}^\top$
		\State \rev{Set $\overline{\underline{\bm{c}}}_{L+1} = \begin{bmatrix}\overline{\underline{\bm{c}}}_L\\\bm{0}\end{bmatrix}+\overline{\underline{\bm{u}}}$}
	\end{algorithmic}
	\caption{Update \emph{with} truncation.}\label{alg:W_trunc}
\end{algorithm}

%\todo[inline]{VS: In line 9 of Alg.5.2 we define
%${\bar c}_{L+1}$ in terms of $c_L$, so $c_L$ is the exact one? The rest of the discussion seems to only address a single approx update, whereas there is an accumulation of errors, if $c_L$ is replaced by ${\bar c}_{L}$}

We may use different error measures for the difference of \Cref{alg:W_trunc,alg:WO_trunc}.
\rev{In the case that only a single new layer $L+1$ is added while centralities for the previous $L$ layers have been computed exactly, the input $\overline{\underline{\bm{c}}}_L = \underline{\bm{c}}_L=(\bm{I} - \bm{\mathcal{M}}_{L})^{-1}\bm{1}_{NL}$ of \Cref{alg:W_trunc} coincides with that of \Cref{alg:WO_trunc}.
We then define the overall relative truncation error}
$$
\frac{\|\overline{\underline{\bm u}} - {\underline{\bm u}}\|_2}{\|{\underline{\bm u}}\|_2},
$$
\rev{and the layer-wise truncation errors}
$$
\|\bm{E}_\ell^\top\underline{\bm{u}} -  \bm{E}_\ell^\top\overline{\underline{\bm{u}}}\|_2 = \|\bm{E}_\ell^\top\underline{\bm{c}}_{L+1} -  \bm{E}_\ell^\top\overline{\underline{\bm{c}}}_{L+1}\|_2, \quad \ell=1,\dots,L+1,
$$
\rev{which decrease for decreasing $\ell$, i.e., the choice of truncation parameter $h$ depends on the rate of error decay.}

\rev{In general, \Cref{alg:W_trunc} may be applied recursively for several layer additions, as required, e.g., in the real-time monitoring of infrastructure networks.
In this case, the input centrality vector $\overline{\underline{\bm{c}}}_L$ already bears a truncation error from previous layer additions.
To measure the accumulation of truncation errors over several layer additions, we define the cumulated error of layer $\ell=1,\dots,L+1$ as}
\begin{equation}\label{eq:cumulated_err}
\frac{\|\overline{\bm{C}}_\ell - \bm C_\ell\|_F}{\|\bm C_\ell\|_F}.
\end{equation}
\rev{where $\bm C_\ell$ and $\overline{\bm{C}}_\ell$ denote the matricizations of the first $\ell$ blocks of the outputs of \Cref{alg:WO_trunc} and \Cref{alg:W_trunc}, respectively.}

\rev{\begin{remark}
	The truncation errors introduced by the repeated application of \Cref{alg:W_trunc} are independent of each other.
	When moving from $L$ to $L+1$ layers, line $1$ of \Cref{alg:W_trunc} computes $\bm{u}_{L+1}=\bm{c}_{L+1}$, i.e., block $L+1$ of $\overline{\underline{\bm{c}}}_{L+1}$ to the employed tolerance.
	Moving from $L+1$ to $L+2$ layers, line $3$ of \Cref{alg:W_trunc} extracts the latter block, which implies that accuracy of the linear system solves in line $6$ is unaffected by prior truncation.
	This means that truncation errors of repeated layer additions are cumulated but not amplified.
	In fact, numerical experiments reported in \Cref{sec:recursive_updating} indicate that, in practice, the cumulation error \eqref{eq:cumulated_err} stagnates after a certain number of layer additions.
\end{remark}}

\subsection{Update of marginal centralities}\label{sec:marginal_update}

The derivation in \Cref{sec:truncation_strategy} also provides an explicit way of updating the marginal indices $\KMN$ and $\KML$ and their truncation as an additional time layer is added.
After $L$ time steps, the failure-aware temporal Katz centrality index for node $i$ is given by
$$
\KMN(i)_L=\frac 1 L {\bm e}_i^T {\bm C}_L {\bm 1}_L ,
$$
where the subscript $L$ keeps track of the number of time steps and $\bm{C}_L\!=\![\bm{c}_1^{(L)},\ldots,\bm{c}_L^{(L)}]$. To determine $\KMN(i)_{L+1}$, the new centrality matrix $\bm{C}_{L+1}=[\bm{c}_1^{(L+1)}, \ldots, \bm{c}_{L+1}^{(L+1)}]$ is obtained from \eqref{rem:c_update} as
\begin{equation}\label{eqn:nextC}
\bm{C}_{L+1}=[\bm{c}_1^{(L)}+\bm{u}_1, \ldots,
\bm{c}_L^{(L)}+\bm{u}_L, \bm{u}_{L+1}],
\end{equation}
where $\bm{u}_{L+1}=\bm{c}_{L+1}^{(L+1)}$. Hence,
$$
\KMN(i)_{L+1}=\frac 1 {L+1}  \left ( L\cdot \KMN(i)_L  + {\bm e}_i^T  \sum_{i=1}^{L+1}\bm{X}_{i,L+1}\bm{1} \right ),
$$
and truncation affects only the summand,
$$
\overline{\KMN(i)}_{L+1}=\frac 1 {L+1}  \left ( L\cdot \KMN(i)_L  + {\bm e}_i^T \sum_{i=L-h}^{L+1}\bm{X}_{i,L+1}\bm{1} \right ).
$$

The marginal layer Katz centrality of layer $\ell=1,\dots,L+1$ follows the same pattern. After $L+1$ time steps and using \eqref{eqn:nextC}, this index for $\ell \le L$ can be written as
\begin{equation*}
\KML(\ell)_{L+1} = \bm{1}_N^\top \bm{C}_{L+1} \bm{e}_\ell =
\KML(\ell)_{L}+ \bm{1}_N^\top {\bm u}_\ell =
\KML(\ell)_{L}+ \bm{1}_N^\top \bm{X}_{\ell,L+1} \bm{1}_N,
\end{equation*}
and can be interpreted as the overall network communicability at time $\ell$.
%; for $\ell=L+1$ we have $\KML(\ell)_{L+1} =\bm{1}_N^\top \bm{X}_{\ell,L+1} \bm{1}_N$ exactly.
Truncation thus leaves $\KML(\ell)_{L+1}$ unaffected for $\ell=L-h,\dots,L+1$, while for $\ell=1,\dots,L-h-1$ the update $\bm{1}_N^\top \bm{X}_{\ell,L+1} \bm{1}_N$ is discarded as it is expected to be negligible \vs{for the employed tolerance}.

%we have the exact quantity at hand since $\bm{1}_N^\top \bm{X}_{\ell,L+1} \bm{1}_N = \bm{1}_N^\top \bm{c}_{L+1}^{(L+1)}$, while for $\ell=1,\dots,L$ it carries a truncation error $\bm{1}_N^\top \bm{X}_{\ell,L+1} \bm{1}_N$ that we have under control via the bound on $\|\bm{X}_{\ell,L+1} \bm{1}_N\|_2$ from \eqref{eq:bound_entries}.
%A finer, entry-wise analogue of \Cref{prop:block_decay}, expressed directly in terms of walk counts as in \cite{arrigo2025updating}, could sharpen this bound to individual entries of $\bm{X}_{\ell,L+1}\bm{1}_N$ rather than its norm; we leave this refinement for future work.

\section{Numerical experiments}\label{sec:numerics}
We now study the truncation and update strategies of \Cref{sec:truncation,sec:truncation_strategy} numerically.
\Cref{sec:data} introduces the test networks and reports the truncation error for three regimes of breakage locality.
\Cref{sec:truncation-numerics} compares the truncated update with the classical Katz index and with exact recomputation, both in accuracy and in runtime.
\Cref{sec:scaling-experiments} closes with the parallel performance of the resulting solvers.
All experiments are carried out in MATLAB R2024b on a dual-socket AMD EPYC~9534 node.
Diagonal-block systems are solved by CG and systems involving $\bm{\mathcal{M}}_L$ by BiCGStab.
% Codes for reproducing all numerical experiments are publicly available under \url{https://github.com/mstoll1602/waternetworks_code}.

\subsection{Test networks}\label{sec:data}

Our test networks are water and power networks of different sizes, symmetrized and binarized if necessary.

For the water network \rev{\texttt{ER}} of Emilia--Romagna with $n=199\,312$ nodes and $N=261\,799$ edges, real breakage and repair records covering four years of daily up- and downdates are available.
These records are reported on the node level and do not distinguish breakages from repairs.
We treat the initial network as intact and interpret every downdate of a previously intact node as a breakage and every update of a previously broken node as a repair.
Since our model operates on the edge level, all edges incident to an up- or downdated node are removed or restored accordingly.

The \texttt{power}\footnote{\url{https://sparse.tamu.edu/PowerSystem/power197k}} network with $n=197\,156$ nodes and $N=320\,812$ edges is taken from the SuiteSparse collection.

The \rev{\texttt{Borgo\_cascade}} water network, a small town in Emilia--Romagna, with $n=9\,821$ nodes and $N=10\,410$ edges is subject to a cascade failure in which, starting from the most central edge, all neighbors of previously downdated edges break in each time step.
This scenario is the worst case with respect to \Cref{prop:block_decay} since consecutive breakages are always at distance $1$.

\subsection{Truncation strategy}

\subsubsection{Illustrative examples}\label{sec:truncation_illustrative}

We first study the effect of the truncation strategy proposed in \Cref{sec:truncation_strategy}, i.e., the difference between \Cref{alg:WO_trunc} and \Cref{alg:W_trunc} for different parameters of $h$.
We include three networks covering different decay behavior with respect to \Cref{prop:block_decay} from very fast over intermediate to very slow decay.
Throughout this section, we use $L=10$ and $c_\alpha=0.8$.

\Cref{tab:truncation_strategy} reports the relative truncation error $\|\overline{\underline{\bm{u}}}-\underline{\bm{u}}\|_2/\|\underline{\bm{u}}\|_2$ together with the layer-wise errors $\|\bm{E}_\ell^\top(\underline{\bm{u}}-\overline{\underline{\bm{u}}})\|_2$, where $\bm{E}_\ell^\top$ extracts the $\ell$-th block of a vector for the three example networks.
Both algorithms were run with the two linear system tolerances $10^{-7}$ and $10^{-12}$ to distinguish truncation errors from linear system errors.
Solutions up to linear system tolerance are denoted by \texttt{tol}.

The three networks represent increasingly demanding regimes in terms of distances between consecutive up- or downdates.
The first $L=10$ layers of real breakage data of the water network \rev{\texttt{ER}} contain between $3$ and $21$ up- or downdates per time step, such that no coupling matrix $\bm{W}_\ell$ vanishes.
In this setting, truncation is not needed at all since solving the block-diagonal systems, i.e., $h=0$ as described in \Cref{sec:truncation} is already accurate up to the linear system tolerance.
For the \texttt{power} network for which we choose $1000$ random breakages per time step, the choices $h=1$ or $h=2$ appear appropriate.
Finally, the cascade failure in \rev{\texttt{Borgo\_cascade}} represents an extreme case in which the truncation error is substantial and decreases only gradually with $h$, as expected when consecutive breakages occur at the minimal possible distance.
In this case, choosing a relatively large truncation parameter may be necessary to obtain the desired accuracy.

\begin{table}
\centering
%\fittable{\begin{tabular}{|c||c|c|c||c|c|c|c}
%\hline\hline \texttt{tol} & \multicolumn{3}{|c|}{$10^{-7}$} &\multicolumn{3}{|c|}{$10^{-12}$}\\\hline\hline
%Error & $h=1$ & $h=2$ & $h=3$ & $h=1$ & $h=2$ & $h=3$\\\hline\hline
%$\frac{\|\overline{\underline{\bm u}} - {\underline{\bm u}}\|_2}{\|{\underline{\bm u}}\|_2}$ & $1.69\cdot 10^{-07}$ & $1.69\cdot 10^{-07}$ & $1.69\cdot 10^{-07}$  & $1.39\cdot 10^{-12}$&$1.39\cdot 10^{-12}$&$1.39\cdot 10^{-12}$\\\hline\hline
%$\|\bm{E}_{L+1}^\top(\underline{\bm{u}} - \overline{\underline{\bm{u}}})\|_2$ & 0 & 0 & 0 & 0 & 0 & 0\\\hline
%$\|\bm{E}_{L}^\top(\underline{\bm{u}} - \overline{\underline{\bm{u}}})\|_2$ & $9.84\cdot 10^{-05}$ &$9.84\cdot 10^{-05}$&$9.84\cdot 10^{-05}$&$8.15\cdot 10^{-10}$&$8.15\cdot 10^{-10}$&$8.15\cdot 10^{-10}$\\\hline
%$\|\bm{E}_{L-1}^\top(\underline{\bm{u}} - \overline{\underline{\bm{u}}})\|_2$ & $1.45\cdot 10^{-07}$ & $1.45\cdot 10^{-07}$ & $1.45\cdot 10^{-07}$ &$1.27\cdot 10^{-12}$&$1.27\cdot 10^{-12}$&$1.27\cdot 10^{-12}$\\\hline
%$\|\bm{E}_{L-2}^\top(\underline{\bm{u}} - \overline{\underline{\bm{u}}})\|_2$ & $0$ & $1.01\cdot 10^{-09}$ & $1.01\cdot 10^{-09}$ & $0$ & $9.55\cdot 10^{-15}$ & $9.55\cdot 10^{-15}$\\\hline
%$\|\bm{E}_{L-3}^\top(\underline{\bm{u}} - \overline{\underline{\bm{u}}})\|_2$ & $0$ & $0$ & $0$ & $0$ & $0$ & $0$\\\hline
%$\|\bm{E}_{L-4}^\top(\underline{\bm{u}} - \overline{\underline{\bm{u}}})\|_2$ & $0$ & $0$ & $0$ & $0$ & $0$ & $0$\\\hline\hline
%\end{tabular}}
\fittable{\begin{tabular}{|c||c|c||c|c|c||c|c|c|c|}
	\hline\hline
	%\texttt{tol} & \multicolumn{3}{|c|}{$10^{-7}$} &\multicolumn{3}{|c|}{$10^{-12}$}\\\hline\hline
	Network & \multicolumn{2}{|c||}{\rev{\texttt{ER}}} & \multicolumn{3}{|c||}{\texttt{power}} & \multicolumn{4}{|c|}{\rev{\texttt{Borgo\_cascade}}}\\\hline\hline
	Error & $h=1$ & $h=2$ & $h=1$ & $h=2$ & $h=3$ & $h=1$ & $h=2$ & $h=3$ & $h=4$\\\hline\hline
	$\frac{\|\overline{\underline{\bm u}} - {\underline{\bm u}}\|_2}{\|{\underline{\bm u}}\|_2}$ & \texttt{tol} & \texttt{tol} & $2.99\cdot 10^{-6}$ & \texttt{tol} & \texttt{tol} & $4.96\cdot 10^{-4}$ & $9.08\cdot 10^{-5}$ & $1.74\cdot 10^{-5}$ & $3.58\cdot 10^{-6}$\\\hline\hline
	$\|\bm{E}_{L+1}^\top(\underline{\bm{u}} - \overline{\underline{\bm{u}}})\|_2$ & 0 & 0 & 0 & 0 & 0 & 0 & 0 & 0 & 0\\\hline
	$\|\bm{E}_{L}^\top(\underline{\bm{u}} - \overline{\underline{\bm{u}}})\|_2$ &\texttt{tol}&\texttt{tol} & \texttt{tol} &\texttt{tol}&\texttt{tol} & \texttt{tol} &\texttt{tol}&\texttt{tol} & \texttt{tol}\\\hline
	$\|\bm{E}_{L-1}^\top(\underline{\bm{u}} - \overline{\underline{\bm{u}}})\|_2$ & \texttt{tol} & \texttt{tol}  & \texttt{tol} &\texttt{tol}&\texttt{tol} & \texttt{tol} &\texttt{tol}&\texttt{tol} & \texttt{tol}\\\hline
	$\|\bm{E}_{L-2}^\top(\underline{\bm{u}} - \overline{\underline{\bm{u}}})\|_2$ & $0$ & \texttt{tol} & $2.0\cdot 10^{-3}$ & \texttt{tol} & \texttt{tol} & $9.99\cdot 10^{-2}$ & \texttt{tol} & \texttt{tol} & \texttt{tol}\\\hline
	$\|\bm{E}_{L-3}^\top(\underline{\bm{u}} - \overline{\underline{\bm{u}}})\|_2$ & $0$ & $0$ & $1.27\cdot 10^{-9}$ & $1.27\cdot 10^{-9}$ & \texttt{tol} & $1.83\cdot 10^{-2}$ & $1.83\cdot 10^{-2}$ & \texttt{tol} & \texttt{tol}\\\hline
	$\|\bm{E}_{L-4}^\top(\underline{\bm{u}} - \overline{\underline{\bm{u}}})\|_2$ & $0$ & $0$ & $0$ & $0$ & $0$ & $3.49\cdot 10^{-3}$ & $3.49\cdot 10^{-3}$ & $3.49\cdot 10^{-3}$ & \texttt{tol}\\\hline
	$\|\bm{E}_{L-5}^\top(\underline{\bm{u}} - \overline{\underline{\bm{u}}})\|_2$ &0&0&0&0&0&$7.11\cdot 10^{-4}$ & $7.11\cdot 10^{-4}$ & $7.11\cdot 10^{-4}$ & $7.11\cdot 10^{-4}$\\\hline\hline
\end{tabular}}
\caption{Relative truncation error and layer-wise absolute truncation errors for adding one layer to the \rev{\texttt{ER}} water network of Emilia--Romagna with real breakage data, the \texttt{power} network with random breakages, and the \texttt{Borgo\_cascade} water network of a small town in Emilia--Romagna with cascade failure, each with $L=10$.
	All errors up to the linear system tolerances $10^{-7}$ and $10^{-12}$ are denoted by \texttt{tol}.}\label{tab:truncation_strategy}
\end{table}

\subsubsection{Truncation accuracy and runtime gains}\label{sec:truncation-numerics}

In addition to the truncation error, we now quantify in how far the truncated centralities deviate from the classical Katz index and what runtime gains are attained by truncation.
All experiments in this subsection use the water network \rev{\texttt{ER}} of Emilia--Romagna with $N=261\,799$ edges, cf.~\Cref{sec:data}, with two modes of up- and downdates.
For \emph{random} breakages, $n_{\mathrm{break}}=6$ edges are chosen uniformly per time step, which matches the mean event rate of the real data.
For \emph{cascade} breakages, the edges are drawn from the line graph neighborhood of the already broken edges, cf.~\Cref{sec:data}.

Using $L=40$ layers, \Cref{fig:trunc-vs-classical} reports the relative truncation error $\|\overline{\underline{\bm{u}}}-\underline{\bm{u}}\|_2/\|\underline{\bm{u}}\|_2$ and the relative error $\|\overline{\underline{\bm{c}}}_k - \K(\bm{A}_k)\bm{1}\|_2/\|\K(\bm{A}_k)\bm{1}\|_2$, where $\overline{\underline{\bm{c}}}_k$ denotes the $k$-th block of the truncated centrality vector obtained by \Cref{alg:W_trunc}.
We consider the interior time points $k=10$ and $k=30$, i.e., time points at which nine and $29$ layers of up- and downdates, respectively, already occurred. We avoid the boundary case $k=1$, for which $h=1$ trivially replaces $\bm{c}_1$ by $\bm{c}_2$ without incurring any truncation error at all.
The error with respect to the untruncated solution drops to machine precision at $h=2$ for random breakages and at $h=2$ to $4$ for cascade breakages.
As predicted by \Cref{prop:block_decay} and already illustrated in \Cref{tab:truncation_strategy}, smaller distances between up- or downdated edges incur larger errors and hence require larger truncation parameters.
The distance to the classical index behaves differently.
For both query rows, it vanishes for $h=0$, where the truncated centrality coincides with the classical index of that row's own network, jumps at $h=1$, and then remains essentially constant up to $h=8$.
The information that distinguishes the failure-aware model of \Cref{sec:FATE} from the classical index is thus already contained in the first block superdiagonal, regardless of which layer is queried.

\begin{figure}[t]
  \centering
  \includegraphics[width=\linewidth]{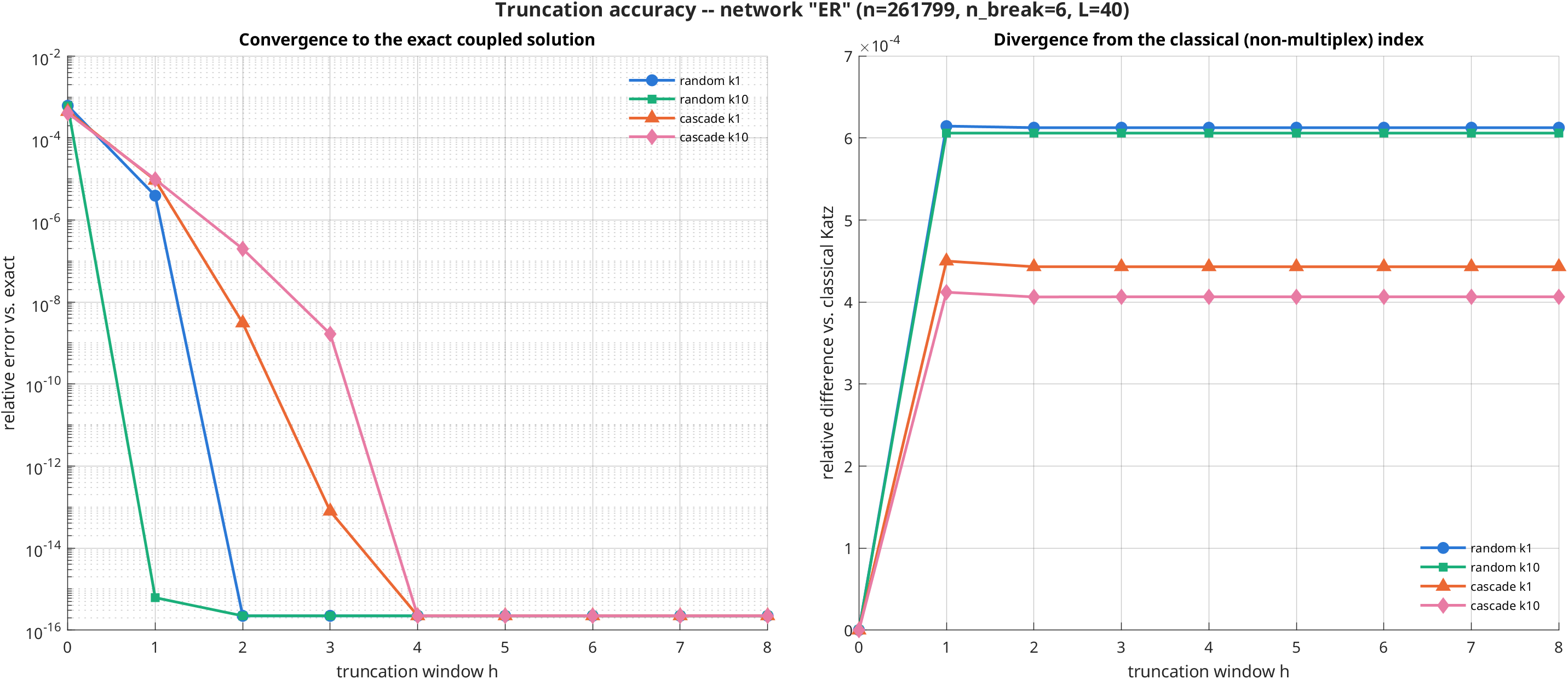}
  \caption{Relative truncation error (left) and relative error w.r.t.\ single-layer Katz centrality (right) for the network \rev{\texttt{ER}} with $L=40$ and $n_{\text{break}}=6$ with random and cascade breakages at the blocks $k=10$ and $k=30$ of the block centrality vector $\K(\bm{\mathcal{M}}_L)\bm{1}$.}
  \label{fig:trunc-vs-classical}
\end{figure}

% Table 5 (tab:truncation-vs-classical) removed: k=10 is already shown in fig:trunc-vs-classical.
% \input{tables/tab_truncation_vs_classical.tex}

\Cref{fig:runtime-comparison} compares the runtime required to obtain $\underline{\bm{c}}_{L+1}$ when adding layer $L+1$ while $\underline{\bm{c}}_L$ is available.
We compare the three strategies: solving the whole $(L+1)$-layer system \eqref{eqn:sys} with BiCGStab from scratch, back-substitution on the same full system, and the truncated update \eqref{eqn:overu} for $h=0,\dots,8$, cf.~\Cref{alg:W_trunc}.
Runtimes of the full BiCGStab solve increase from $4.6$ to $69.1$ seconds as $L$ grows from $20$ to $300$ while back-substitution runtimes increase from $0.38$ to $3.7$ seconds.
In contrast, the cost of the truncation strategy remains between $0.09$ and $0.22$ seconds independently of $L$ for different values of $h$.
The speedup over the full solve therefore increases with $L$, from a factor of $45$ at $L=20$ to a factor of $679$ at $L=300$ for $h=0$, and from $22$ to $316$ for $h=8$, cf.~\Cref{tab:comparsupdate}.

\begin{figure}[t]
  \centering
  \subfloat[Runtimes w.r.t.\ $L$]{
  \includegraphics[width=.55\linewidth]{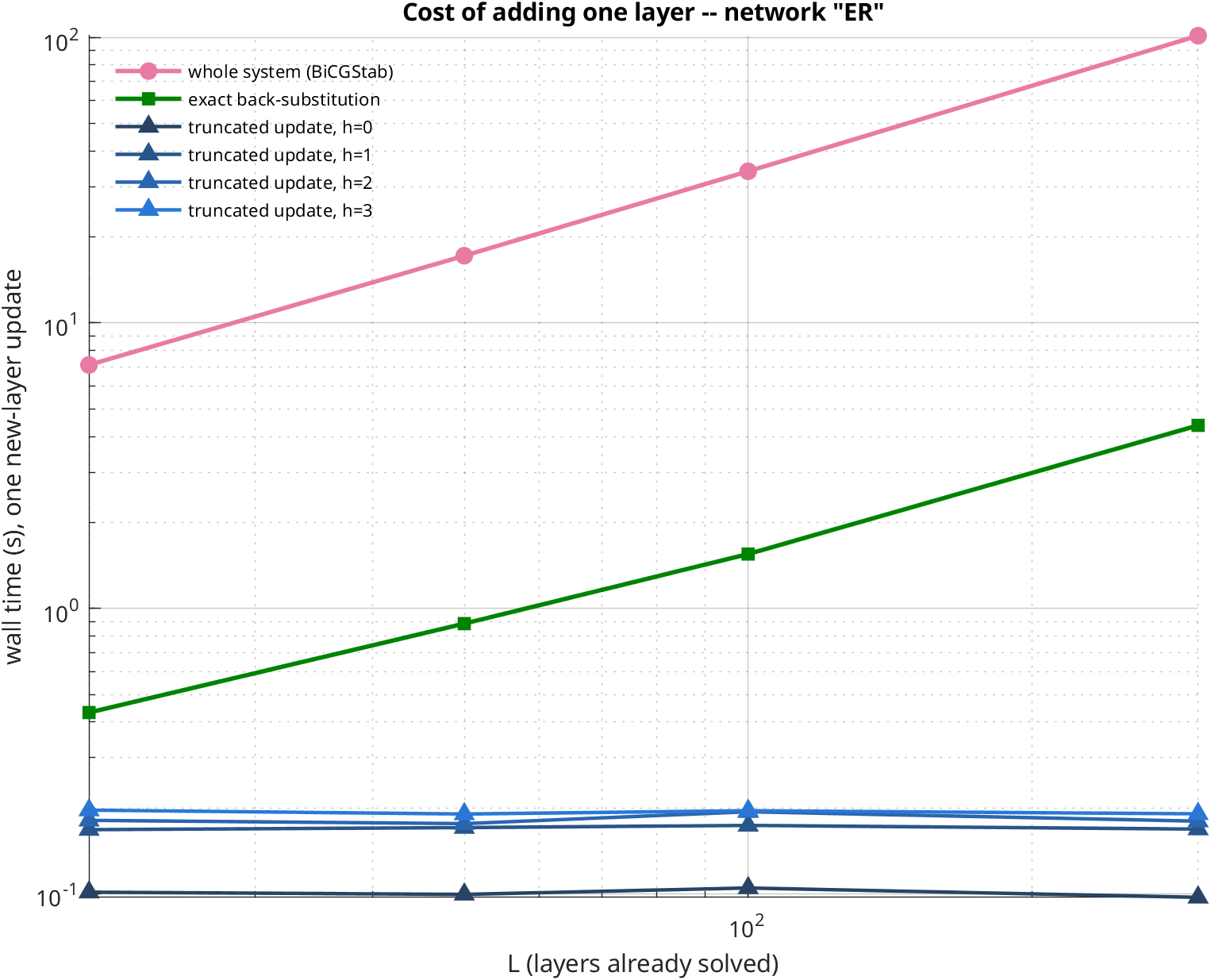}\label{fig:runtime-comparison}}
\hfill
	\subfloat[Runtimes for $h=8$]{%\fittable{%
			\footnotesize
			\begin{tabular}{lrr}
				\toprule
				L & Truncation (s) & BICGstab (s) \\
				\midrule
				20 & 0.2138 & 4.6198 \\
				50 & 0.2098 & 12.4471 \\
				100 & 0.2137 & 23.9885 \\
				300 & 0.2185 & 69.1402 \\
				\bottomrule
	\end{tabular}\label{tab:comparsupdate}}
%}
  \caption{Runtimes to obtain $\underline{\bm{c}}_{L+1}$ once layer $L+1$ is added to an untruncated $L\times L$-block system for the \rev{\texttt{ER}} (Emilia--Romagna) water network:
  whole-system BiCGStab and exact back-substitution (both recomputing all $L+1$ layers from scratch) against the truncation strategy with $h=0,\dots,8$, cf.~\Cref{alg:W_trunc}.}
\end{figure}

% Table 6 (tab:comparsionapproaches) removed: same data as fig:runtime-comparison.
% \input{tables/tab_comparsionapproaches.tex}

%\input{tables/tab_comparsupdate.tex}

\subsubsection{Cumulated truncation error for recursive updating}\label{sec:recursive_updating}

Finally, starting from $L=1$, \Cref{fig:trajectory} recursively adds $100$ layers with different truncation parameters.
We use $n_{\mathrm{break}}=100$ random edge breakages per step.
We measure the cumulated error $\|\overline{\bm{C}}_\ell-\bm{C}_\ell\|_F/\|\bm{C}_\ell\|_F$ after $\ell$ updates, where $\bm{C}_\ell,\overline{\bm{C}}_\ell\in\R^{N\times \ell}$ collect the untruncated and truncated centrality vectors of all layers $1,\dots,\ell$.
With random breakages, distances between consecutive breakages are sufficiently large such that \rev{$\K(\bm{\mathcal{M}}_L)$ is approximately block-diagonal, which implies that} the error of all truncation strategies ranges at machine precision over the whole trajectory.
With cascade breakages, truncation errors grow during the first approximately $40$ updates, peaking at $1.0\cdot 10^{-6}$ for $h=1$ and $2.5\cdot 10^{-8}$ for $h=2$.
The truncation strategy hence exhibits bounded cumulated errors, even in the worst case scenario.
%Under cascade breakages the adaptive window invests more solves than either fixed window when the data demands it ($335$, versus $199$ for $h=1$ and $297$ for $h=2$), but in exchange nearly halves the error of $h=2$ and cuts the error of $h=1$ by a factor of $84$.
% \Cref{tab:truncupdate-extended} lists the cumulated error at several points of the trajectory.

\begin{figure}[t]
  \centering
  \includegraphics[width=\linewidth]{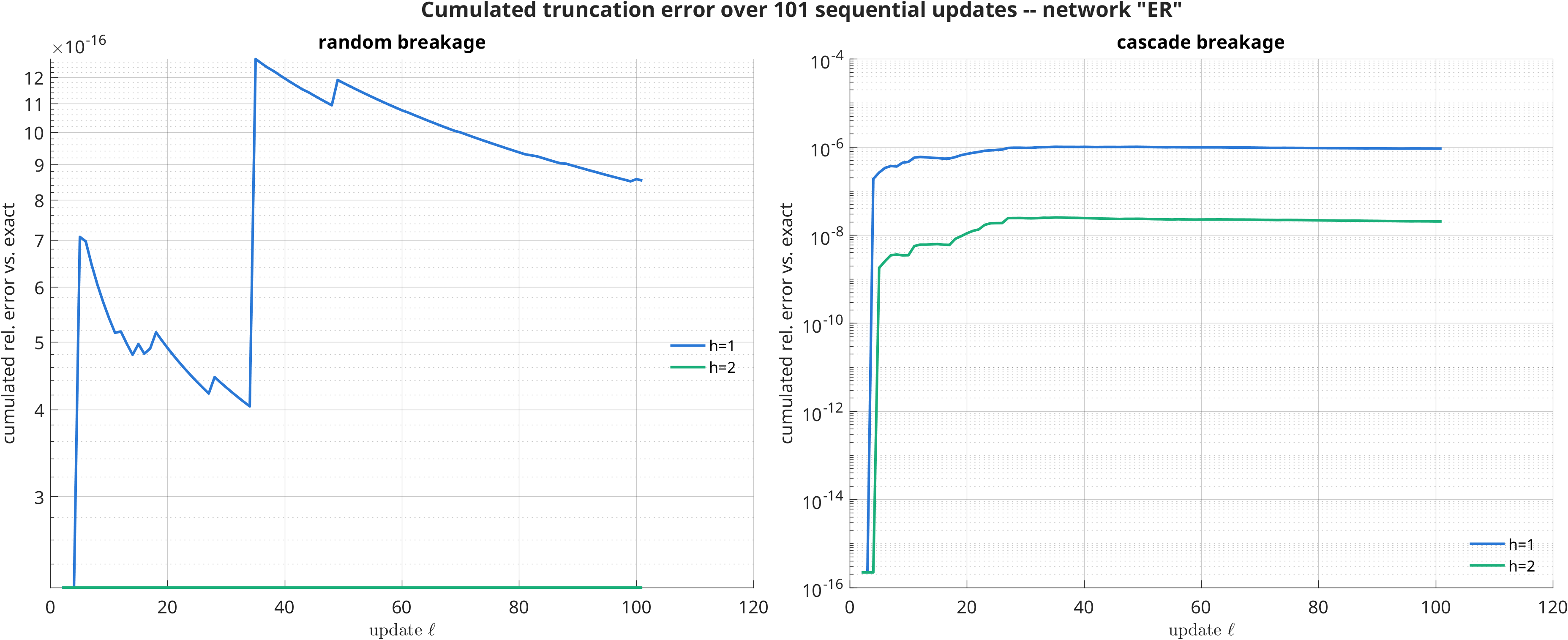}
  \caption{Cumulated relative error over $L=101$ sequential updates for the network \rev{\texttt{ER}} ($n_{\mathrm{break}}=100$ random edge breakages per step) with $h=1$ (blue) and $h=2$ (green), for random (left) and cascade (right) breakages.}
  \label{fig:trajectory}
\end{figure}

\subsection{Parallel solvers and scaling}\label{sec:scaling-experiments}

The block structure of \eqref{centrality} offers two sources of parallelism.
First, truncation after the super-diagonal as described in \Cref{sec:truncation}, which corresponds to $h=0$ and to which we refer as block CG.
As discussed in \Cref{sec:truncation}, the incurred truncation error depends on distances between consecutive up- or downdates.
Second, in the real-time monitoring of infrastructure networks, one is typically interested in the potential effect of several candidate breakages for the next time step.
We represent $R\in\N$ such candidates via the different adjacency matrices $\bm{A}_{L+1}^{(1)},\dots,\bm{A}_{L+1}^{(R)}$ for layer $L+1$.
The computed quantities hence only differ by the right-most quantity $\bm{u}_{L+1}^{(r)}=\K(\bm{A}_{L+1}^{(r)})\bm{1}$ while $\bm{c}_L$ and $\bm{W}_{L-h},\dots,\bm{W}_{L-1}$ are identical.
We refer to the concurrent evaluation of $R$ such candidates as the \rev{\emph{truncated-update ensemble}}.

To reduce communication costs, both strategies apply $(\bm{I}-\bm{A}_\ell)$ without ever forming $\bm{A}_\ell$ explicitly.
Instead, we define a diagonal matrix $\bm{D}_\ell$ with diagonal entries $0$ and $1$ to mark edges that are intact in layer $\ell$.
Matrix-vector products are hence realized via $\bm{A}_\ell = \alpha\bm{D}_\ell\widetilde{\bm{A}}\bm{D}_\ell$.
Every worker now stores $\widetilde{\bm{A}}$ once instead of one matrix per layer, which for $L=768$ layers of the network \rev{\texttt{ER}} reduces the data per worker from $14$\,GB to $18$\,MB, which makes the approach scalable to large-scale problems.

We use the water network \rev{\texttt{ER}} of Emilia--Romagna with $n_{\mathrm{break}}=100$ random breakages per time step, $c_\alpha=0.8$, a CG tolerance of $10^{-8}$, and the truncation window $h=2$ in the truncated-update ensemble setting.
All timings report the minimum of three runs on up to $p=64$ workers, i.e., one socket of the node.
Following standard methodology \cite{amdahl1967,gustafson1988}, we report the wall clock time $T(p)$ on $p$ workers together with the parallel efficiency $E(p)=T(1)/(p\,T(p))$ for a fixed problem size (strong scaling), and $E(p)=T(1)/T(p)$ for a problem size that grows proportionally with $p$ (weak scaling).
Both efficiencies are measured against the one-worker runtime of the same parallel code, since the thread usage of a serial reference implementation cannot be controlled reliably.

\Cref{tab:strong-scaling,tab:weak-scaling,fig:scaling} report parallel scaling results over a range of $p=1$ to $p=64$, corresponding to $L=R=12$ to $L=R=784$ layers on a dual-socket AMD EPYC~9534 node.
Both strategies exhibit an efficiency of at least $84\%$ up to $p=16$ for both weak and strong scaling and reach $60\%$ for block CG and $65\%$ for the truncated ensemble at $p=64$.
\Cref{tab:strong-scaling} shows that the matrix-free approach offers runtime advantages over communicating all layer adjacency matrices with its advantage growing from a factor of $1.12$ at $p=1$ to $3.20$ at $p=64$.
Absolute timings illustrate the computational benefit of both strategies.
BiCGStab applied to the full system with $NL\approx 2.0\cdot 10^{8}$ requires $261.1$ seconds, while the decoupled diagonal blocks are sequentially solved in $62.2$ seconds.
A single truncated update requires $0.20$ seconds, such that the allocation of $12$ layers to each worker leads to $2.6$ seconds for an ensemble of $768$ candidate breakages on $64$ workers.
This corresponds to a runtime two orders of magnitude lower compared to one solve of the full system.

Finally, the bottom-right panel of \Cref{fig:scaling} illustrates memory contention: all cores on the node share one path to RAM, such that competing for the shared path slows down each core when running many of them in parallel.
A workload whose data fits inside each core's own local cache barely requires that shared path, which ensures efficiency regardless of how many other cores are busy.
The plot shows the slowdown factor $T(p)/T(1)$ where every one of the $p$ workers repeats the exact same fixed task independently.
Tt equals $1$ if adding more simultaneous workers has no effect, while larger values above $1$ indicate stronger competition for memory.

\begin{table}[t]
  \centering
  \footnotesize
  	\caption{Runtimes and parallel efficiency in the strong scaling setting with fixed $L=R=768$ for the network \rev{\texttt{ER}}.
  	Efficiency is computed against the pool's own $p=1$ time.
  	``Truncated ensemble'' times an independent candidate breakage scenario while ``block CG'' solves one diagonal block system per worker.}
  \label{tab:strong-scaling}
  \fittable{
  \begin{tabular}{lrrrrrrr}
    \toprule
    Workers $p$                    & 1     & 2     & 4     & 8     & 16    & 32    & 64    \\
    \midrule
    Block CG, matrix-free (s)      & 44.24 & 24.94 & 12.22 & 6.08  & 3.22  & 1.84  & 1.16  \\
    Block CG, explicit (s)         & 49.76 & 27.26 & 14.61 & 7.73  & 4.94  & 4.25  & 3.72  \\
    Truncated ensemble (s)         & 108.33& 58.77 & 29.75 & 14.81 & 7.90  & 4.26  & 2.62  \\
    \midrule
    Strong efficiency, block CG           & 1.00  & 0.89  & 0.91  & 0.91  & 0.86  & 0.75  & 0.60  \\
    Strong efficiency, truncated ensemble & 1.00  & 0.92  & 0.91  & 0.91  & 0.86  & 0.79  & 0.65  \\
    \bottomrule
  \end{tabular}}
\end{table}

\begin{table}[t]
  \centering
  \footnotesize
  \caption{Runtimes and parallel efficiency in the weak scaling setting with $L=R=12p$ for the network \rev{\texttt{ER}}.
  	``Truncated ensemble'' times an independent candidate breakage scenario while ``block CG'' solves one diagonal block system per worker.}
  \label{tab:weak-scaling}
  \fittable{
  \begin{tabular}{lrrrrrrr}
    \toprule
    Workers $p$                    & 1     & 2     & 4     & 8     & 16    & 32    & 64    \\
    Layers $L=R$                   & 12    & 24    & 48    & 96    & 192   & 384   & 768   \\
    \midrule
%    \multicolumn{8}{l}{\emph{Weak scaling}}\\
    Block CG (s)                   & 0.74  & 0.76  & 0.80  & 0.80  & 0.86  & 0.91  & 1.08  \\
    Truncated ensemble (s)         & 1.68  & 1.97  & 1.86  & 1.86  & 2.01  & 2.10  & 2.42  \\\midrule
    Weak efficiency, block CG      & 1.00  & 0.97  & 0.92  & 0.92  & 0.85  & 0.82  & 0.68  \\
    Weak efficiency, truncated ensemble    & 1.00  & 0.85  & 0.90  & 0.90  & 0.84  & 0.80  & 0.69  \\\bottomrule
    % Contention diagnostic (identical fixed workload run on every worker) -- left out, see git history.
    % \midrule
    % \multicolumn{8}{l}{\emph{Contention diagnostic (identical fixed workload run on every worker)}}\\
    % Sparse solve slowdown          & 1.00  & 1.13  & 1.14  & 1.15  & 1.17  & 1.22  & 1.54  \\
    % In-cache control slowdown      & 1.00  & 1.00  & 1.00  & 1.00  & 1.01  & 1.04  & 1.09  \\
  \end{tabular}}
\end{table}

\begin{figure}[t]
  \centering
  \includegraphics[width=\linewidth]{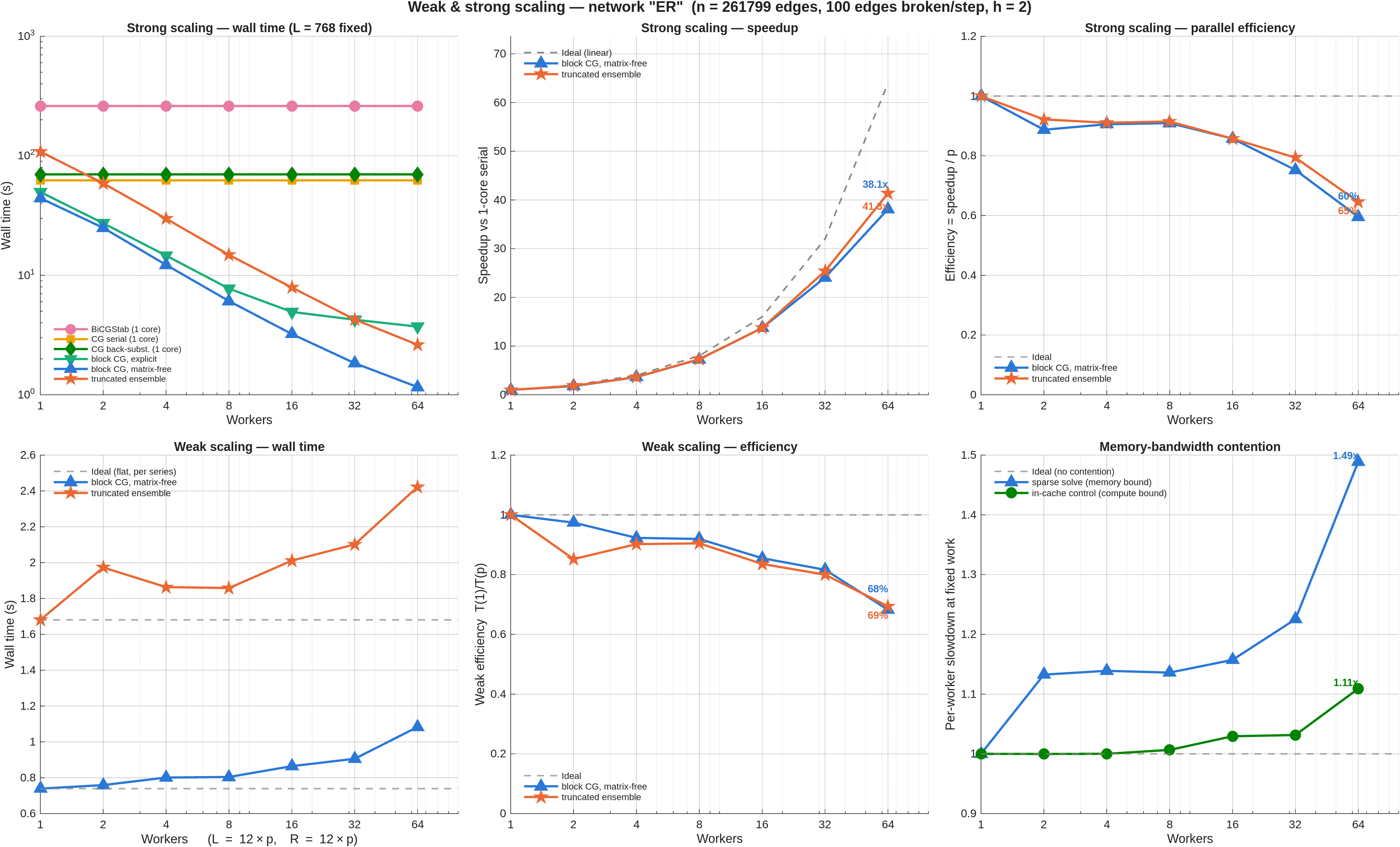}
  \caption{Weak and strong scaling of the block-CG ($h=0$) and truncated-update ($h=2$) approaches on the water network \rev{\texttt{ER}} with $N=261\,799$ edges and $100$ random edge downdates per layer with up to $p=64$ workers (one socket).
  	Top row: strong-scaling wall time, speedup,
    and parallel efficiency at fixed $L=R=768$.
    Bottom row: weak-scaling wall
    time and efficiency at $L=R=12p$ as well as the memory-bandwidth contention
    diagnostic comparing a
    memory-bound sparse solve against a compute-bound in-cache control for an
    identical task per worker.}
    %Efficiency exceeds $1$ for $p\le 32$ because splitting the fixed workload across more workers shrinks each worker's data below the point where it fits in cache, so the resulting speedup outpaces $p$ until bandwidth contention pulls efficiency back below $1$ at $p=64$.}
  \label{fig:scaling}
\end{figure}

%%%%%%%%%%%%%
\section{Conclusions}

We have proposed a failure-aware temporal multiplex network model relying on line graph adjacency matrices and low-rank update matrices in a supra-adjacency framework.
The analysis of the inverse of this block matrix in terms of entry decay allowed for two highly attractive computational approaches based on truncating the block-inverse with rigorous error control.

We believe that computationally efficient numerical linear algebra frameworks for edge-based analysis of complex networks represent a promising road for future research.
Especially classes of networks not considered in this manuscript, such as small-world networks, appear to be challenging in this respect.
%\todo[inline]{any other future work?}

\vs{
\section*{Acknowledgments} F.G. and V.S. would like to thank HeraTech S.r.l. for providing the data used to generate the matrices for the examples {\tt{ER}} and {\tt{Borgo\_cascade}}.
The work of F.G. was performed while at the Department of Mathematics, University of Bologna, as a PhD student, supported by the European Union - NextGenerationEU
under the National Recovery and Resilience Plan (PNRR) DM351 (09/04/2022), Mission 4, Component 1, Investments 3.4 and 4.1.
V.S. is a member of the INdAM Research Group
GNCS, and her work was partially supported by the European Union - NextGenerationEU
under the National Recovery and Resilience Plan (PNRR) - Mission 4 Education and research -
Component 2 From research to business - Investment 1.1 Notice Prin 2022 - DD N. 104 of 2/2/2022, code 20227PCCKZ – CUP J53D23003620006.
}

The authors thank Sebastian Esche for providing the data for the Mittweida street network.
\bibliographystyle{siam}
\bibliography{mlwater}

%%%%%%%%%%%%%%%%%%%%%%%%%%%%%%%%%%%%%%%%%
\appendix

\section{Proof of \Cref{prop:block_inverse}}\label{proof:block_inverse}
% This appendix contains the proof of Proposition~\ref{prop:X}.

% Proof begins
\begin{proof}
The matrix $\bm{\mathcal{M}}_L$ is block upper bidiagonal, hence $\K(\bm{\mathcal{M}}_L) = ( \bm{I}-\bm{\mathcal{M}}_L)^{-1}$ is block upper triangular, with block diagonal entries $\bm{X}_{k,k} = \K(\bm{A}_k) =  (\bm{I}-\bm{A}_k )^{-1}$. We only need to derive the upper blocks by back substitution. Moreover, for any block column $\ell$, the entries of $\bm{X}_{.,\ell}$ only depend on the principal matrix $\bm{\mathcal{M}}_\ell,$ where $\bm{\mathcal{M}}_\ell$ describes the upper left block matrix up to time-point $\ell$ for all $\ell=1,\ldots,L$. Hence, we can show how to write the elements of the last block of columns, and the other ones will have the same structure. The matrix $\bm{X}_{.,L}$ is obtained as the solution to the following linear system
\[
( \bm{I}-\bm{\mathcal{M}}_L) \bm{X}_{.,L} = \bm{E}_{L}, \quad \bm{E}_{L} = [\bm{0}, \dots, \bm{0}, \bm{I}_N]^T \in \mathbb{R}^{NL \times N}
\]

% Matrix representation
that is
\[
\begin{bmatrix}
\bm{I}-\bm{A}_1 & -\bm{W}_1 & \bm{0} & \dots & \bm{0} \\
\bm{0} & \bm{I}-\bm{A}_2 & -\bm{W}_2 & \dots & \bm{0} \\
\bm{0} & \bm{0} & \ddots & \ddots & \vdots \\
\bm{0} & \bm{0} & \dots & \bm{I}-\bm{A}_{L-1} & -\bm{W}_{L-1} \\
\bm{0} & \bm{0} & \dots & \bm{0} & \bm{I}-\bm{A}_{L}
\end{bmatrix}
\begin{bmatrix}
\bm{X}_{1,L} \\
\bm{X}_{2,L} \\
\vdots \\
\bm{X}_{L-1,L} \\
\bm{X}_{L,L}
\end{bmatrix}
=
\begin{bmatrix}
\bm{0} \\
\bm{0} \\
\vdots \\
\bm{0} \\
\bm{I}_N
\end{bmatrix}.
\]

Using $\bm{X}_{L,L} = (\bm{I}-\bm{A}_{L})^{-1} = \K(\bm{A}_L)$ and back substituting, we obtain
\[
( \bm{I}-\bm{\mathcal{M}}_{L-1} ) 
\begin{bmatrix}
\bm{X}_{1,L} \\
\vdots \\
\bm{X}_{L-1,L}
\end{bmatrix}
= \bm{E}_{L-1} \bm{W}_{L-1} (\bm{I}-\bm{A}_{L})^{-1}.
\]

Hence, by solving with $\bm{\mathcal{M}}_{L-1}$ using the same strategy, we obtain
$$
\bm{X}_{L-1,L} = (\bm{I}-\bm{A}_{L-1})^{-1} \bm{W}_{L-1} (\bm{I}-\bm{A}_{L})^{-1} = \K(\bm{A}_{L-1}) \bm{W}_{L-1} \K(\bm{A}_{L}),
$$
and substituting again,
\begin{align}
(\bm{I}-\bm{\mathcal{M}}_{L-2} )
\begin{bmatrix}
\bm{X}_{1,L} \\
\vdots \\
\bm{X}_{L-2,L}
\end{bmatrix}
&= \bm{E}_{L-2} \bm{W}_{L-2} \bm{X}_{L-1,L} \\
&= \bm{E}_{L-2} \bm{W}_{L-2} (\bm{I}-\bm{A}_{L-1})^{-1} \bm{W}_{L-1} (\bm{I}-\bm{A}_{L})^{-1}.
\end{align}

Solving the system above explicitly gives the next block: again, we explicitly obtain the bottom block term as
\begin{align*}
\bm{X}_{L-2,L} &~= ( \bm{I}-\bm{A}_{L-2})^{-1} \bm{W}_{L-2} (I-\bm{A}_{L-1})^{-1} \bm{W}_{L-1} (\bm{I}-\bm{A}_{L})^{-1}\\
&~= \K(\bm{A}_{L-2}) \bm{W}_{L-2} \K(\bm{A}_{L-1}) \bm{W}_{L-1} \K(\bm{A}_{L})
\end{align*}
and so on, up to the block $\bm{X}_{1,L}$.
\end{proof}

\section{Proof of \Cref{prop:block_decay}}
\label{app:proof_prop_4.4}

\begin{proof}
The proof proceeds in four steps: first, we factorize the disruption matrices $\bm{W}_m$; second, we bound the norm of the core interaction blocks; third, we bound the boundary terms; and finally, we assemble these bounds to derive the result.

\textbf{Step 1: Factorization of the Block Product.}
Recall from Proposition 4.1 that for $k < \ell$, the block $\bm{X}_{k,\ell}$ of the inverse matrix $(\bm{I} - \mathcal{\bm{M}}_L)^{-1}$ is given by
\begin{equation}
\bm{X}_{k,\ell} = \left( \prod_{m=k}^{\ell-1} \K(\bm{A}_m) \bm{W}_m \right) \K(\bm{A}_\ell),
\end{equation}
where $\K(\bm{A}) = (\bm{I} - \bm{A})^{-1}$. 
Each disruption matrix $\bm{W}_m$ represents a rank-2 modification corresponding to the 
\vs{removal or inclusion of}
%breakage or repair of a pipe at
node $i_m$. We can factorize $\bm{W}_m$ as
\begin{equation}
\bm{W}_m = \bm{U}_m \widehat{\bm{U}}_m^T, \quad \text{with} \quad \bm{U}_m = \begin{bmatrix} e_{i_m} & a_{i_m} \end{bmatrix}, \quad \widehat{\bm{U}}_m = \begin{bmatrix} a_{i_m} & e_{i_m} \end{bmatrix},
\end{equation}
where $a_{i_m} = \bm{A}_m e_{i_m}$ is the $i_m$-th column of $\bm{A}_m$. 
Substituting this factorization into the expression for $\bm{X}_{k,\ell}$, we obtain
\begin{equation}
\bm{X}_{k,\ell} = \K(\bm{A}_k) \bm{U}_k \left( \prod_{m=k}^{\ell-2} \widehat{\bm{U}}_m^T \K(\bm{A}_{m+1}) \bm{U}_{m+1} \right) \widehat{U}_{\ell-1}^T \K(\bm{A}_\ell).
\end{equation}
For notational convenience, let us define the $2 \times 2$ \textit{core interaction matrix} at step $m$ as
\begin{equation}
\bm{T}_m := \widehat{\bm{U}}_m^T \K(\bm{A}_{m+1}) \bm{U}_{m+1} \in \mathbb{R}^{2 \times 2}.
\end{equation}
Then the action of $\bm{X}_{k,\ell}$ on the vector of ones $\mathbf{1}$ can be written as
\begin{equation}
\bm{X}_{k,\ell} \mathbf{1} = \K(\bm{A}_k) \bm{U}_k \left( \prod_{m=k}^{\ell-2} \bm{T}_m \right) \widehat{\bm{U}}_{\ell-1}^T \K(\bm{A}_\ell) \mathbf{1}.
\end{equation}
Note that $\K(\bm{A}_\ell) \mathbf{1} = c_\ell$ is the Katz centrality vector of layer $\ell$.

\textbf{Step 2: Bounding the Core Interaction Matrix.}
We now bound the norm of $\bm{T}_m$. The entries of $\bm{T}_m$ are of the form $u^T \K(\bm{A}_{m+1}) v$, where we point out that $u, v \in \{e_{i_m}, a_{i_m}, e_{i_{m+1}}, a_{i_{m+1}}\}$. 
Specifically, the entries involve terms like $e_{i_m}^T \K(\bm{A}_{m+1}) e_{i_{m+1}}$ and $a_{i_m}^T \K(\bm{A}_{m+1}) a_{i_{m+1}}$.
By \Cref{lem:lemma2}  the entries of $\K(\bm{A}_{m+1})$ decay exponentially with the graph distance. Specifically, for any nodes $p, q$:
\begin{equation}
|[\K(\bm{A}_{m+1})]_{pq}| \leq C \lambda^{d(p,q)},
\end{equation}
where $\lambda = \frac{\sqrt{1+c_\alpha} - \sqrt{1-c_\alpha}}{\sqrt{1+c_\alpha} + \sqrt{1-c_\alpha}}$ and $C$ is the constant from \Cref{lem:lemma2}.
Since $a_{i_m} = \bm{A}_m e_{i_m}$ is a linear combination of neighbors of $i_m$, the distance from any neighbor $i_m'$ to $i_{m+1}$ satisfies $|d(i_m, i_{m+1}) - d(i_m', i_{m+1})| \leq 1$ as discussed in \Cref{lemma:lemma_1}. 
For terms involving $a_{i_m} = \bm{A}_m e_{i_m}$, we have:
\begin{align}
|a_{i_m}^T \K(\bm{A}_{m+1}) e_{i_{m+1}}| 
&= \left| \sum_{j \in \mathcal{N}(i_m)} \alpha [\K(\bm{A}_{m+1})]_{j, i_{m+1}} \right| \\
&\leq \sum_{j \in \mathcal{N}(i_m)} \alpha C \lambda^{d(j, i_{m+1})} \\
&\leq \alpha g_{i_m} C \lambda^{d(i_m, i_{m+1}) - 1},
\end{align}
where the last inequality uses \Cref{lemma:lemma_1}: $d(j, i_{m+1}) \geq d(i_m, i_{m+1}) - 1$ for all $j \in \mathcal{N}(i_m)$.
Thus, terms involving $a_{i_m}$ introduce a factor proportional to the degree $g_{i_m}$ and the decay rate $\lambda^{d(i_m, i_{m+1}) - 1}$.
Combining these bounds, each entry of $\bm{T}_m$ satisfies
\begin{equation}
|[\bm{T}_m]_{rs}| \leq (1 + \alpha g_{\max}) C \lambda^{d(i_m, i_{m+1}) - 1},
\end{equation}
where $g_{\max}$ is the maximum degree of the network. 
Using the Frobenius norm to bound the spectral norm ($\|\bm{T}_m\|_2 \leq \|\bm{T}_m\|_F$), and noting $\bm{T}_m$ is $2 \times 2$, we obtain
\begin{equation}
\|\bm{T}_m\|_2 \leq 2 (1 + \alpha g_{\max}) C \lambda^{d(i_m, i_{m+1}) - 1}.
\end{equation}
Let $\tau_m := 2 (1 + \alpha g_{\max}) C \lambda^{d(i_m, i_{m+1}) - 1}$. Then $\|\bm{T}_m\|_2 \leq \tau_m$.

\textbf{Step 3: Bounding the Boundary Terms.}
We bound the left and right boundary terms in the product chain.
\textit{Left Boundary:} The term $\|\K(\bm{A}_k) \bm{U}_k\|_2$ involves the columns $\K(\bm{A}_k) e_{i_k}$ and $\K(\bm{A}_k) a_{i_k}$. 
Since $\|\K(\bm{A}_k)\|_2 \leq (1 - c_\alpha)^{-1}$ and $\|a_{i_k}\|_2 \leq \alpha \sqrt{g_{i_k}} \leq \alpha \sqrt{g_{\max}}$, we have
\begin{equation}
\|\K(\bm{A}_k) \bm{U}_k\|_2 \leq \sqrt{2} \max \left\{ \|\K(\bm{A}_k) e_{i_k}\|_2, \|\K(\bm{A}_k) a_{i_k}\|_2 \right\} \leq \frac{\sqrt{2} (1 + \alpha \sqrt{g_{\max}})}{1 - c_\alpha}.
\end{equation}
For simplicity, we bound this by $\frac{\sqrt{2} (1 + \alpha g_{\max})}{1 - c_\alpha}$.

\textit{Right Boundary:} The term is $\|\widehat{\bm{U}}_{\ell-1}^T c_\ell\|_2$. Since $c_\ell = \K(\bm{A}_\ell) \mathbf{1}$, the entries of $\widehat{\bm{U}}_{\ell-1}^T c_\ell$ are $a_{i_{\ell-1}}^T c_\ell$ and $e_{i_{\ell-1}}^T c_\ell$. 
Let $[c_\ell]_{\max}$ be the maximum entry of the centrality vector $c_\ell$. Then
\begin{equation}
\|\widehat{\bm{U}}_{\ell-1}^T c_\ell\|_2 \leq \sqrt{1 + (\alpha g_{\max})^2} [c_\ell]_{\max} \leq (1 + \alpha g_{\max}) [c_\ell]_{\max}.
\end{equation}

\textbf{Step 4: Assembly of the Bound.}
Combining the results from Steps 1--3, we apply the sub-multiplicativity of the vector norm:
\begin{equation}
\|\bm{X}_{k,\ell} \mathbf{1}\|_2 \leq \|\K(\bm{A}_k) \bm{U}_k\|_2 \left( \prod_{m=k}^{\ell-2} \|\bm{T}_m\|_2 \right) \|\widehat{\bm{U}}_{\ell-1}^T c_\ell\|_2.
\end{equation}
Substituting the bounds derived above:
\begin{equation}
\|\bm{X}_{k,\ell} \mathbf{1}\|_2 \leq \left( \frac{\sqrt{2} (1 + \alpha g_{\max})}{1 - c_\alpha} \right) \left( \prod_{m=k}^{\ell-2} \tau_m \right) \left( (1 + \alpha g_{\max}) [c_\ell]_{\max} \right).
\end{equation}
Grouping the constants and the decay terms, we define $\bar{C} = \frac{\sqrt{2} (1 + \alpha g_{\max})^2}{1 - c_\alpha} [c_\ell]_{\max}$. The product of the decay terms yields
\begin{equation}
\prod_{m=k}^{\ell-2} \lambda^{d(i_m, i_{m+1}) - 1} = \lambda^{\sum_{m=k}^{\ell-2} (d(i_m, i_{m+1}) - 1)}.
\end{equation}
Adjusting the constant to absorb the factors of $2$ and $C$ from $\tau_m$, we arrive at the final bound stated in Proposition 4.4:
\begin{equation}
\|\bm{X}_{k,\ell} \mathbf{1}\|_2 \leq \bar{C} \lambda^{\sum_{m=k}^{\ell-2} (d(i_m, i_{m+1}) - 1)},
\end{equation}
where $\bar{C}$ aggregates the polynomial factors in $g_{\max}$, $c_\alpha$, and $[c_\ell]_{\max}$.
\end{proof}

\end{document}